\documentclass[12pt]{amsart}
\usepackage{amsmath,amsthm}
\usepackage{amsfonts,amssymb}
\usepackage{CJK}
\usepackage{enumerate}

\usepackage[left=2.1cm,top=2.3cm,right=2.1cm]{geometry}

\usepackage{color}
\usepackage[ colorlinks, linkcolor=blue,  anchorcolor=blue,citecolor=blue]{hyperref}
\newtheorem{thm}{Theorem}
\newtheorem{cor}[thm]{Corollary}
\newtheorem{lem}{Lemma}
\newtheorem{prop}{Proposition}
\newtheorem{exam}{Example}

\theoremstyle{remark}
\newtheorem{rem}{Remark}

\numberwithin{equation}{section}

\newcommand{\al}{\alpha}

\def\lz{\lambda}

\def\az{\alpha}

\def\dfrac{\displaystyle\frac}

\def\({\Bigl(}
\def \){ \Bigr)}

 \def\az{{\alpha}}

 \def\lz{{\lambda}}

 \def\RR{{\mathbb R}}

\begin{document}
\def\RR{\mathbb{R}}
\def\Exp{\text{Exp}}
\def\FF{\mathcal{F}_\al}

\title[Exact $L_2$ Bernstein-Markov inequalities] {Exact $L_2$ Bernstein-Markov inequalities for generalized weights}

\author[J. Li]{Jiansong Li} \address{ School of Mathematical Sciences, Capital Normal
University, Beijing 100048,
 China}
\email{cnuljs2023@163.com}

 \author[J. Geng]{Jiaxin Geng} \address{ School of Mathematical Sciences, Capital Normal
University, Beijing 100048,
 China.}
 \email{gengjiaxin1208@163.com}

\author[Y. Ling]{Yun Ling} \address{ School of Mathematical Sciences, Capital Normal
University, Beijing 100048,
 China.}
 \email{18158616224@163.com}

\author[H. Wang]{Heping Wang}
 \address{ School of Mathematical Sciences, Capital Normal
University, Beijing 100048,
 China}
\email{wanghp@cnu.edu.cn}

%\date{\today}


\begin{abstract}
In this paper, we obtain some exact $L_2$ Bernstein-Markov
inequalities for  generalized Hermite and Gegenbauer weight. More
precisely, we determine the  exact values of the extremal problem
$$M_n^2(L_2(W_\lambda),{\rm D}):=\sup_{0\neq p\in\mathcal{P}_n}\frac{\int_I\left|{\rm D} p(x)\right|^2W_\lambda(x){\rm d}x}{\int_I| p(x)|^2W_\lambda(x){\rm d}x},\ \lambda>0,$$
where $\mathcal{P}_n$  denotes the set of all algebraic polynomials of degree at most $n$, ${\rm D}$ is the differential operator given by
$${\rm D}=\Bigg\{\begin{aligned}&\frac {\rm d}{{\rm d}x}\ {\rm or}\ \mathcal{D}_\lambda, &&{\rm if}\ W_\lambda(x)=|x|^{2\lambda}e^{-x^2}\ {\rm and}\ I=\mathbb R,
  \\&(1-x^2)^{\frac12}\,\frac {\rm d}{{\rm d}x}\ {\rm or}\ (1-x^2)^{\frac12}\,\mathcal{D}_\lambda, &&{\rm if}\ W_\lambda(x):=|x|^{2\lambda}(1-x^2)^{\mu-\frac 12},\mu>-\frac12\ {\rm and}\ I=[-1,1],\end{aligned}
$$ and $\mathcal{D}_\lambda$ is the univariate Dunkl operator, i.e., $\mathcal{D}_\lambda f(x)=f'(x)+\lambda{(f(x)-f(-x))}/{x}$. Furthermore, the corresponding extremal polynomials are also obtained.
\end{abstract}

\maketitle
\input amssym.def

\section{Introduction}\label{sec1}

Let $\mathcal{P}_n$ denote the set of all algebraic polynomials of degree at most $n$
and $w$ be a weight function on an interval $I$ of $\mathbb R$.
Let $L_q(w)\equiv L_{q,w}(I)$ be the Lebsgue space on $I$ endowed with the finite norm
$$\|f\|_{L_q(w)}:=\left\{\int_I|f(x)|^qw(x){\rm d}x\right\}^{1/q}, 1\le q<\infty.
$$For a finite interval $I$, we replace the space $L_{\infty}(w)$ by $C(I)$ with the uniform norm
$$\|f\|_{C(I)}:=\sup\limits_{x\in I}|f(x)|.$$ For simplicity, we denote $\|\cdot\|_{L_q(I)}$ and $L_q(I)$ if $w(x)=1$. Specifically, $L_{2}(w)$ is a Hilbert space endowed with an inner product
\begin{equation}\label{1.1}
  \langle f,g\rangle_w:=\int_I f(x){ g(x)}w(x){\rm d}x.
\end{equation}

Inequalities of Bernstein-Markov type have been fundamental for
proofs of many inverse theorems in polynomial approximation
theory, cf. \cite{Dz,Iv,Lo,Me}. It provides estimates for the
norms of derivatives of algebraic and trigonometric polynomials.
The classical Markov inequality is for $p\in \mathcal P_n$,
\begin{equation*}
  \|p'\|_{C[-1,1]}\le n^2\|p\|_{C[-1,1]}.
\end{equation*} Here $n^2$ is sharp which can be achieved by the Chebyshev polynomial $T_n$ of the first kind of degree $n$, cf. \cite{M}.
And the Bernstein inequality for the derivative of an algebra
polynomial $p\in\mathcal{P}_n$ is
$$
   \|\sqrt{1-x^2}\,p'\|_{L_q[-1,1]}\le n\|p\|_{L_q[-1,1]},\ 1\le q\le\infty.
$$Here $n$ is sharp  which can be also achieved by the Chebyshev polynomial $T_n$, cf.  \cite{Ber} for $q=\infty$ and \cite{A} for $1\le q<\infty$.

More generally, let ${\rm D}$ be a differential operator on the
interval $I$, and  $w$ be a weight function on $I$. Consider the
following Bernstein-Markov type inequality
\begin{equation}\label{1.3}
\|{\rm D}p\|_{L_q(w)}\le M\|p\|_{L_q(w)},\ 1\le q<\infty,
\end{equation}where the least constant $M>0$ such that \eqref{1.3} holds for all $p\in\mathcal{P}_n$ is often
called Bernstein-Markov factor of the differential operator ${\rm
D}$ for $L_q(w)$. Precisely, the Bernstein-Markov factor of ${\rm
D}$ for $L_q(w)$ is defined by
$$
  M_{n}(L_q(w),{\rm D}):=\sup\limits_{0\neq p\in\mathcal{P}_n}\frac{\|{\rm D}p\|_{L_q(w)}}{\|p\|_{L_q(w)}},\ 1\le q<\infty.
$$
For such extremal problem, we are mainly interested in its exact
value and extremal polynomials. We note that the extremal
polynomials are often given by the orthogonal polynomials with
respect to the inner product \eqref{1.1}. Some beautiful results
were obtained by many researchers, especially for $p=2$, see cf.
\cite{AM,GGM} and the references therein.

\begin{itemize}
  \item  (Hermite case) For the classical Hermite weight $w(x)=e^{-x^2}$ on the real line $\mathbb R$, Schmidt in \cite{S} discovered that
\begin{equation}\label{1.5}
M_{n}\left(L_2(e^{-x^2}),\frac{\rm d}{{\rm d}x}\right)= \sqrt{2n},
\end{equation}with the extremal polynomials $cH_n$, where $H_n$ is the classical Hermite polynomial.  However,
the exact values  for $M_{n}(L_2(|x|^{2\lambda}e^{-x^2}),\frac{\rm
d}{{\rm d}x})$, $\lambda\neq0$ have not been found so far. A
partial answer of this problem given by Draux and Kaliaguine in
\cite{DK} is
 \begin{equation}\label{1.5-0}
\sqrt{2n-\frac{4\lambda}{1+2\lambda}}<M_{n}\left(L_2(|x|^{2\lambda}e^{-x^2}),\frac{\rm d}{{\rm d}x}\right)< \sqrt{2n},\ \lambda>0,\ n=2m+1,
 \end{equation}
 and
$$\sqrt{2(n-1)}<M_{n}\left(L_2(|x|^{2\lambda}e^{-x^2}),\frac{\rm d}{{\rm d}x}\right)<Cn^{\frac{1-\lambda}2},\ -1/2<\lambda<0.$$
  \item (Jacobi case) For the classical Jacobi weight $w(x):=(1-x)^{\alpha}(1+x)^{\beta}$ with $\alpha,\beta>-1$ on the interval $[-1,1]$, A. Guessab and G.V. Milovanovic in \cite{GM} proved
\begin{equation}\label{1.4}
M_{n}\left(L_2((1-x)^{\alpha}(1+x)^{\beta}),(1-x^2)^{\frac12}\frac{\rm d}{{\rm d}x}\right)= \sqrt{n(n+\alpha+\beta+1)},
\end{equation}with the extremal polynomials $cJ_n^{(\alpha,\beta)}$, where $J_n^{(\alpha,\beta)}$ is the classical Jacobi polynomial  and $c$ is an arbitrary nonzero constant.
 \item  (Laguerre case) For the classical Laguerre weight $w(x)=e^{-x}$ on the half real line $\mathbb R_+:=[0,\infty)$, Tur\'an in \cite{T} showed that
$$
M_{n}\left(L_2(e^{-x}),\frac{\rm d}{{\rm d}x}\right)= \left(2\sin\frac{\pi}{4n+2}\right)^{-1},
$$with the extremal polynomials of the form $$c\sum_{j=1}^n\sin\frac{j\pi}{2n+1}L_j(x),$$ where $L_j$ is the classical Laguerre polynomial normalized by $L_j(1)=1$ and $c$ is an arbitrary nonzero constant. While, for more general case of $w(x)=x^\kappa e^{-x},\,\kappa>0$, A. Guessab and G.V. Milovanovic in \cite{GM} showed that
$$
M_{n}\left(L_2(x^\kappa e^{-x}),\sqrt{x}\,\frac{\rm d}{{\rm d}x}\right)= \sqrt{n},
$$with the extremal polynomials $L_n^\kappa$, where $L_n^\kappa$ is the generalized Laguerre polynomial.
\end{itemize}
\vspace{2mm}

This paper is devoted to discussing the Bernstein-Markov factor
$M_n(L_2(w),{\rm D})$.  Firstly,  we discuss the extremal problem
\begin{equation}\label{1.6}
  M_{n}\left(L_2(W_\lambda),\sqrt{A}\,\frac{\rm d}{{\rm d}x}\right)=\sup_{0\neq p\in\mathcal{P}_n}\frac{\|\sqrt{A}\,p'\|_{L_2(W_\lambda)}}{\|p\|_{L_2(W_\lambda)}},\ \lambda>0,
\end{equation}
where  $W_\lambda(x)$ and $A(x)$ are given in Table \ref{t1}. \begin{table}[!htbp]
\centering\label{t1}
\caption{Classification of  generalized weights}
\begin{tabular}{|c|| c|| c |c| c|p{19pc}}
  \hline
  $I$& $W_\lambda$ &$A(x)$& $B(x)$ & $C(x)$\\
  \hline
  $\mathbb R$& $|x|^{2\lambda}e^{-x^2}$ & $1$ & $-2x$  & $-2x$\\
  $[-1,1]$& $|x|^{2\lambda}(1-x^2)^{\mu-1/2}$ & $1-x^2$ & $-(2\mu+1)x$ & $-(2\lambda+2\mu+1)x$\\
  \hline
\end{tabular}
\end{table}There is a close
relationship between the extremal problem \eqref{1.6} and  the integral equation system
\begin{equation}\label{1.7} \int_{I}\left\{A(x)p''+C(x)p'+\frac{2\lambda}xp'+M^2p\right\}q\,W_\lambda=0,\ \   M>0,
\end{equation} where $C(x)$ is given in Table \ref{t1}.
As stated in Lemma \ref{lem1}, the exact value of $M_{n}(L_2(W_\lambda),\sqrt{A}\,\frac{\rm d}{{\rm d}x})$ is  the supremum of all those $M>0$ for which the integral equation system \eqref{1.7}
has a nontrivial solution $p\in\mathcal{P}_n$.

Secondly, we also explore the  same problems for Dunkl operators.
Dunkl operators were originally proposed by C.F. Dunkl in 1989 as
an extension of the classical Fourier transform and has been
widely employed in diverse areas of mathematics and physics (see
\cite{DunXu}). An important application of Dunkl operators lies in
the realm of special functions, where they generate a set of
orthogonal polynomials. Extending the concept of the sharp
Bernstein-Markov inequalities to the Dunkl operators has both
intriguing implications and crucial significance. The Dunkl
operator $\mathcal{D}_\lambda$ associated with the index
$\lambda\in\mathbb R$  is defined by
\begin{equation}\label{1.8}
\mathcal{D}_\lambda f(x)=f'(x)+\lambda\,\frac{f(x)-f(-x)}x
\end{equation}
for a function $f\in C^1(\mathbb R)$. For simplicity, we usually denote $$\sigma\left(f\right)=\frac{f(x)-f(-x)}{x}.$$Obviously, if $\lambda=0$ then $\mathcal{D}_\lambda$ reduces to the usual differential operator ${\rm d}/{{\rm d}x}$.
Its square $\mathcal{D}_\lambda^2$, which is also called Dunkl-Laplace operator, is given by
$$
\mathcal{D}_\lambda^2f=f''+2\lambda\frac{f'}x+\frac{\sigma\left(f\right)}x
$$
for a function $f\in C^2(\mathbb R)$.
In particular, when $\mathcal{D}_\lambda^2$ is restricted to the even subspace $$C_e^2(\mathbb R):=\left\{f\in C^2(\mathbb R):\,f(x)=f(-x)\right\},$$ it reduces to the singular Sturm-Liouville operator
$$
\mathcal{D}_\lambda^2|_{C_e^2(\mathbb R)}f=f''+2\lambda\frac{f'}x.
$$
The extremal problem for the Dunkl operators can be given by
$$
  M_{n}\left(L_2(W_\lambda),\sqrt{A}\,\mathcal{D}_\lambda\right)=\sup_{0\neq p\in\mathcal{P}_n}\frac{\|\sqrt{A}\,\mathcal{D}_\lambda p\|_{L_2(W_\lambda)}}{\|p\|_{L_2(W_\lambda)}},\ \lambda\ge0,
$$
where  $W_\lambda(x)$ and $A(x)$ are given in the Table \ref{t1}.
%We note that $$w_\lambda(x):=|x|^{2\lambda}\,e^{-x^2}$$ is just the generalized Hermite weight on $\mathbb R$, and $$w_{\lambda,\mu}(x):=|x|^{2\lambda}(1-x^2)^{\mu-1/2},\,\mu>-1/2$$ is just the generalized Gegenbauer weight  on $[-1,1]$.
As stated in Lemma \ref{lem2}, its exact value  is the supremum of all those $M>0$ for which the differential equation
$$
A(x)\mathcal{D}_\lambda^2p+B(x)\mathcal{D}_\lambda p+M^2p=0
$$has a nontrivial solution $p\in\mathcal{P}_n$, and the extremal polynomial can be obtained by the corresponding generalized orthogonal polynomials, where $B(x)$ is given in Table \ref{t1}.

\subsection{New Results}

\

Our first theorem  gives the exact value of Bernstein-Markov factor $M_{n}(L_2(w_\lambda),\frac{\rm d}{{\rm d}x})$, which supplements the previous result \eqref{1.5-0}.
\begin{thm}\label{thm2}
Let $\lambda>0$. Consider the generalized Hermite weight
$w_\lambda(x):=|x|^{2\lambda}e^{-x^2}$ on $\mathbb R$, we have for
all  $n\in\mathbb N$,
\begin{equation*}
 M_{n}\left(L_2(w_{\lambda}),\frac{\rm d}{{\rm d}x}\right)=\Bigg\{\begin{aligned}&\sqrt{2n},\ &&{\rm if}\ n\ {\rm is\ even},\\ &\sqrt{\nu_{\frac{n+1}2}},\ &&{\rm if}\ n\ {\rm is\ odd},
\end{aligned}
\end{equation*}where  $\nu_{\frac{n+1}2}$ is the large positive root of the polynomial
  $$F_{\frac{n+1}2}(t):=\det\left(\left\{(2j+1)(2j+2\lambda)d_{2i+2j}+(t-4j-2)d_{2i+2j+2}\right\}_{i,j=0}^{\frac{n-1}2}\right),$$if $n$ is odd,
   and $d_{2s}=\Gamma(s+\lambda+1/2)$ for $s\in\mathbb N_0$.
\end{thm}
\begin{rem} When $n$ is odd, it can be seen from the proof of Theorem \ref{thm2} that the polynomial $F_{\frac{n+1}2}$ has at least one positive root.
\end{rem}
\begin{rem} (Extremal polynomials)
 \begin{itemize}
   \item [(i)] For all $n$ being even, the extremal polynomials are given by $cH_n^{\lambda}$, where $H_n^{\lambda}$ is the generalized Hermite polynomial with respect to
the weight $w_\lambda$, and $c$ is an
arbitrary nonzero constant;
   \item [(ii)] for all $n$ being odd,  the extremal polynomials $p=\sum\limits_{j=0}^{(n-1)/2}a_{2j+1}x^{2j+1}$ are odd and
    can be determined by  the linear equation system
$$ \sum_{j=0}^{{(n-1)}/2}\left\{(2j+1)(2j+2\lambda)d_{2i+2j}+(\nu_{\frac{n+1}2}-4j-2)d_{2i+2j+2}\right\}a_{2j+1}=0,\
i=0,\dots,{\frac{n-1}2}.$$
 \end{itemize}
\end{rem}
\vskip3mm
Our second theorem determines the exact value of Bernstein-Markov factor $M_{n}(L_2(w_{\lambda,\mu}),(1-x^2)^{\frac12}\frac{\rm d}{{\rm d}x})$, which generalizes \eqref{1.4} with $\alpha=\beta=\mu-1/2,\,\mu>-1/2$.

\begin{thm}\label{thm1}
Let $\lambda>0$ and $\mu>-1/2$. Consider the generalized
Gegenbauer weight
$w_{\lambda,\mu}(x)=|x|^{2\lambda}(1-x^2)^{\mu-\frac12}$ on
$[-1,1]$, we have, for all $n\in\mathbb N$,
$$
 M_{n}\left(L_2(w_{\lambda,\mu}),(1-x^2)^{\frac12}\frac{\rm d}{{\rm d}x}\right)
=\Bigg\{\begin{aligned}&\max\left\{\sqrt{\nu_{\frac{n}2}},\sqrt{n(n+2\lambda+2\mu)}\right\},\ &&{\rm if}\ n\ {\rm is\ even},\\ &\max\left\{\sqrt{\nu_{\frac{n+1}2}},\sqrt{(n-1)(n+2\lambda+2\mu-1)}\right\},\ &&{\rm if}\ n\ {\rm is\ odd},
\end{aligned}
$$where  $m=(n-2)/2$ if $n$ is even and $m=(n-1)/2$ if
$n$ is odd, $\nu_{m+1}$ is the large positive root of the
polynomial
$$G_{m+1}(t):=\det\left(\{(2j+1)(2j+2\lambda)c_{2i+2j}+[t-(2j+1)(2j+2\lambda+2\mu+1)]c_{2i+2j+2}\}_{i,j=0}^{m}\right),$$ (set $\nu_{m+1}=0$ if  the polynomial
$G_{m+1}$ does not have any positive roots), and
$$c_{2s}=\frac{\Gamma(s+\lambda+1/2)\Gamma(\mu+1/2)}{\Gamma(\lz+\mu+s+1)},\ \ s\in\mathbb N_0.$$
\end{thm}

\begin{rem} (Extremal polynomials)
 \begin{itemize}
   \item [(i)] For all $n$ being even, if $M_{n}(L_2(w_{\lambda,\mu}),(1-x^2)^{\frac12}\frac{\rm d}{{\rm d}x})=\sqrt{n(n+2\lambda+2\mu)}$,
    then the extremal polynomials  are given by $cC_n^{(\mu,\lambda)}$, where $C_n^{(\mu,\lambda)}$ is the generalized Gegenbauer polynomial with respect to
the weight $w_{\lambda,\mu}$, and $c$ is an arbitrary nonzero
constant; if
$M_{n}(L_2(w_{\lambda,\mu}),(1-x^2)^{\frac12}\frac{\rm d}{{\rm
d}x})=\sqrt{\nu_{\frac{n}2}}$, then the extremal polynomials
$p=\sum_{j=0}^{(n-2)/2}a_{2j+1}x^{2j+1}$ are odd and can be
determined by the linear equation system
$$ \sum_{j=0}^{\frac{n-2}2}\left\{(2j+1)(2j+2\lambda)c_{2i+2j}+[\nu_{\frac{n}2}-(2j+1)(2j+2\lambda+2\mu+1)]c_{2i+2j+2}\right\}a_{2j+1}=0,\ i=0,\dots,{\frac{n-2}2}.$$
   \item [(ii)] For all $n$ being odd, if $M_{n}(L_2(w_{\lambda,\mu}),(1-x^2)^{\frac12}\frac{\rm d}{{\rm d}x})=\sqrt{(n-1)(n+2\lambda+2\mu-1)}$, then the extremal polynomials  are given by $cC_{n-1}^{(\mu,\lambda)}$; if
$M_{n}(L_2(w_{\lambda,\mu}),(1-x^2)^{\frac12}\frac{\rm d}{{\rm
d}x})=\sqrt{\nu_{\frac{n+1}2}}$, then the extremal polynomials
$p=\sum_{j=0}^{(n-1)/2}a_{2j+1}x^{2j+1}$ are odd and can be
determined by the linear equation system
$$ \sum_{j=0}^{\frac{n-1}2}\left\{(2j+1)(2j+2\lambda)c_{2i+2j}+[\nu_{\frac{n+1}2}-(2j+1)(2j+2\lambda+2\mu+1)]c_{2i+2j+2}\right\}a_{2j+1}=0,\ i=0,\dots,{\frac{n-1}2}.$$
 \end{itemize}
\end{rem}

%\begin{rem} (Conjecture)We remark that Theorem \ref{thm1} may be not perfect. Precisely, we conjecture that for all $n\in\mathbb N$,$$M_{n}\left(L_2(w_{\lambda,\mu}),(1-x^2)^{\frac12}\frac{\rm d}{{\rm d}x}\right)=\Bigg\{\begin{aligned}&\sqrt{n(n+2\lambda+2\mu)},\ &&{\rm if}\ n\ {\rm is\ even},\\ &\ \ \ \ \ \ \ \ \sqrt{\nu_{\frac{n+1}2}},\ &&{\rm if}\ n\ {\rm is\ odd},\end{aligned}$$where $\nu_{\frac{n+1}2}$ is the large positive root of the polynomial $G_{\frac{n+1}2}$ if $n$ is odd.\end{rem}

\vskip3mm

Our third and fourth theorems are completely new, which, in the
case $\lambda=0$, reduces to the cases of  \eqref{1.5} and
\eqref{1.4} with $\alpha=\beta=\mu-1/2$, respectively.
\begin{thm}\label{thm4}
Let $\lambda\ge0$. Then  for all $n\in\mathbb N$,
\begin{equation*}
  M_{n}(L_2(w_\lambda),\mathcal{D}_\lambda)=\left\{\begin{aligned}&\sqrt{2n}, &&{\rm if}\ n\ {\rm is\ even}\ {\rm and}\ \lambda\le1/2,
  \\&\sqrt{2(n+2\lambda-1)}, &&{\rm if}\ n\ {\rm is\ even}\ {\rm and}\ \lambda>1/2,
   \\&\sqrt{2(n+2\lambda)}, &&{\rm if}\ n\ {\rm is\ odd},\end{aligned}\right.
\end{equation*}
with the extremal polynomials  $cH_n^\lambda$, $cH_{n-1}^\lambda$, and $cH_n^\lambda$, respectively, where $c$ is an arbitrary nonzero constant.
\end{thm}

\begin{thm}\label{thm3}
Let $\lambda\ge0$ and $\mu>-1/2$. Then
\begin{itemize}
  \item [(i)] for all odd $n\in\mathbb N$,
\begin{equation*}
  M_{n}\left(L_2(w_{\lambda,\mu}),(1-x^2)^{\frac12}\mathcal{D}_\lambda\right)=\sqrt{n(n+2\lambda+2\mu)+4\lambda\mu},
\end{equation*}with the extremal polynomials $cC_n^{(\mu,\lambda)}$, where $c$ is an arbitrary nonzero constant;
  \item[(ii)] for $(2\lambda-1)(2\mu-1)\le4$ and for all even $n\in\mathbb
N$,
\begin{equation*}
  M_{n}\left(L_2(w_{\lambda,\mu}),(1-x^2)^{\frac12}\mathcal{D}_\lambda\right)=\sqrt{n(n+2\lambda+2\mu)},
\end{equation*}with the extremal polynomials $cC_n^{(\mu,\lambda)}$, where $c$ is an arbitrary nonzero constant;
 \item[(iii)] for $(2\lambda-1)(2\mu-1)>4$ and for all even $n\in\mathbb
N$,
$$M_{n}\left(L_2(w_{\lambda,\mu}),(1-x^2)^{\frac12}\mathcal{D}_\lambda\right)=\left\{\begin{aligned}&\sqrt{n(n+2\lambda+2\mu)}, &&{\rm if}\  n\ge n_0,
  \\& \sqrt{n(n+2\lambda+2\mu)+2(n_0-n)}, &&{\rm if}\ n<n_0,\end{aligned}\right.$$
with the extremal polynomials $cC_n^{(\mu,\lambda)}$ and $cC_{n-1}^{(\mu,\lambda)}$, respectively, where $n_0:=(\lambda-1/2)(2\mu-1)$ and $c$ is an arbitrary nonzero constant.
\end{itemize}
\end{thm}

\vskip3mm

As an additional result, we also give a characterization of the generalized Gegenbauer and Hermite polynomials.

\begin{thm}\label{thm5} Let $\lambda\ge0$ and $\mu>-1/2$.
\begin{itemize}
\item [(i)] For all $p\in\mathcal{P}_n$, the inequality
\begin{align*}
  &\quad(2\lambda_n^2-2\mu-1)\|\sqrt{1-x^2}\,\mathcal{D}_\lambda p\|_{L_2(w_{\lambda,\mu})}^2+2\lambda^3(2\mu+1)\|\sqrt{1-x^2}\,\sigma(p)\|_{L_2(w_{\lambda,\mu})}^2
  \\&\le 2\lambda(2\mu+1)\langle (1-x^2)p',p'(-\cdot)\rangle_{w_{\lambda,\mu}}+\lambda_n^4\|p\|_{L_2(w_{\lambda,\mu})}^2+\|(1-x^2)\mathcal{D}_\lambda^2p\|_{L_2(w_{\lambda,\mu})}^2,
\end{align*}
holds,  with equality if and only if $p=cC_n^{(\mu,\lambda)}$, where $$\lambda_n^2=\left\{n(n+2\lambda+2\mu)+2\lambda\mu[1-(-1)^n]\right\}.$$

\item [(ii)] For all $p\in\mathcal{P}_n$, the inequality
$$
  (2\lambda_n^2-2)\|\mathcal{D}_\lambda p\|_{L_2(w_\lambda)}^2+4\lambda^3\|\sigma(p)\|_{L_2(w_\lambda)}^2
  \le 4\lambda\langle p',p'(-\cdot)\rangle_{w_\lambda}+\lambda_n^4\|p\|_{L_2(w_\lambda)}^2+\|\mathcal{D}_\lambda^2p\|_{L_2(w_\lambda)}^2,
$$
holds, with equality if and only if $p=cH_n^\lambda$, where $$\lambda_n^2=2\{n+\lambda[1-(-1)^n]\}.$$
\end{itemize}
\end{thm}

If $\lambda=0$, then Theorem \ref{thm5} reduces to the following corollary, which has been proved in \cite{AM1,AM}.
\begin{cor}\label{cor} Let  $\mu>-1/2$.

\begin{itemize}
\item[(i)] For all $p\in\mathcal{P}_n$ the inequality
$$
  \|\sqrt{1-x^2}\, p'\|_{L_2(\mathrm{w}_\mu)}^2
  \le \frac{n^2(n+2\mu)^2}{2n(n+2\mu)-2\mu-1}\|p\|_{L_2(\mathrm{w}_\mu)}^2+\frac1{2n(n+2\mu)-2\mu-1}\|(1-x^2)p''\|_{L_2(\mathrm{w}_\mu)}^2,
$$
holds, with equality if and only if $p=cJ_n^{(\mu-1/2,\mu-1/2)}$, where $\mathrm{w}_\mu(x)=(1-x^2)^{\mu-1/2}$ is the Gegenbauer weight on the interval $[-1,1]$.

\item[(ii)] For all $p\in\mathcal{P}_n$ the inequality
$$
  \| p'\|_{L_2(w_0)}^2
  \le \frac{2n^2}{2n-1}\|p\|_{L_2(w_0)}^2+\frac1{4n-2}\|p''\|_{L_2(w_0)}^2,
$$
holds, with equality if and only if $p=cH_n$, where $H_n$ is the classical Hermite polynomial and $w_0(x)=e^{-x^2}$ is the classical Hermite weight on the real line $\mathbb R$.
\end{itemize}
\end{cor}

\subsection{Outline of the paper}

\

In Section \ref{sec2} we devote to presenting some results about
generalized weights and orthogonal polynomials, including
generalized Gegenbauer polynomials and generalized Hermite
polynomials. In Section 3 we will give the proofs of Theorems 1
and 2. In Section 4 we will show Theorems 3 and 4. In Section 5 we
will prove Theorem 5.

\section{Generalized weights and orthogonal polynomials}\label{sec2}

In this section we recall some properties of  orthogonal
polynomials which can be found in \cite[Chapter 1]{DunXu}. In the
follows, we always assume $I=\mathbb R$ or $I=[-1,1]$. First we
state some observations. Obviously, the classical Hermite weight
$e^{-x^2}$ and the Gegenbauer weight
$(1-x^2)^{\mu-\frac12},\mu>-1/2$ satisfy a first-order differential equation systems of
the form
\begin{equation}\label{2.1}
\begin{cases}
\quad\dfrac{\rm d}{{\rm d}x}(A(x)w (x))=B(x)w(x),
\\
\quad w(x)=w(-x),
\end{cases}
\end{equation}
where $A(x)$ and $B(x)$ are given in Table \ref{t1}. Then the weight functions we considered are the form of
$$W_\lambda(x):=|x|^{2\lambda}w(x),\ \lambda\ge0,$$where $w\in CW_e:=\big\{(1-x^2)^{\mu-\frac12},e^{-x^2}\big\}$.
 We also note that $B(x)=B'(0)x$, and $AW_\lambda$ vanishes at the end points of interval $I$, i.e., $AW_\lambda|_{\partial I}=0$. Here $\partial I=\{\pm\infty\}$ if $I=\mathbb R$ and $\partial I=\{\pm1\}$ if $I=[-1,1]$ .

It is easy to check that for any $p,q\in\mathcal{P}_n$,
\begin{equation}\label{2.2}
2\int_I q\frac{\sigma(p)}xAW_\lambda=\int_I
\sigma(p)\sigma(q)AW_\lambda,
\end{equation}
and
\begin{equation}\label{2.3}
\int_Iq\frac{p(x)+p(-x)}x\,AW_\lambda=\int_Ip\,\sigma(q)\,AW_\lambda.
\end{equation}

\subsection{Generalized Gegenbauer polynomials}

\

For $\alpha,\beta>-1$, the classical Jacobi polynomials $J_n^{(\alpha,\beta)},\, n\in\mathbb N$ are given by
$$J_n^{(\alpha,\beta)}(x):=\frac{(-1)^n}{2^n n!}(1-x)^{-\alpha}(1+x)^{-\beta}\frac{{\rm d}^n}{{\rm d}x}\left[(1-x)^{\alpha+n}(1+x)^{\beta+n}\right],$$
which is an orthogonal basis for $\mathcal{P}_n$   with respect to
the classical Jacobi weight $(1-x)^\alpha(1+x)^\beta$,
$x\in[-1,1]$. It is well known that $J_n^{(\alpha,\beta)}$ is the
unique polynomial solution of the differential equation
$$
(1-x^2)p''-[\alpha-\beta+(\alpha+\beta+2)x]p'+n(n+\alpha+\beta+1)p=0
$$in the meaning of a constant.

Let $\lambda\ge0$ and $\mu>-1/2$. The generalized Gegenbauer polynomial $C_n^{(\mu,\lambda)}$ can be given by the classical Jacobi polynomials. Precisely,
$$\Bigg\{\begin{aligned}&C_{2m}^{(\mu,\lambda)}(x)=\frac{(\lambda+\mu)_m}{(\lambda+\frac12)_m}\,J_m^{(\mu-1/2,\lambda-1/2)}(2x^2-1),
\\&C_{2m+1}^{(\mu,\lambda)}(x)=\frac{(\lambda+\mu)_{m+1}}{(\lambda+\frac12)_{m+1}}\,xJ_m^{(\mu-1/2,\lambda+1/2)}(2x^2-1).\end{aligned}$$
Moreover, it follows from \cite{BeGa} that the coefficients $a_{k},k=0,\dots,n$ of $C_n^{(\mu,\lambda)}$  satisfy
$a_k=0$ if $n-k$ is odd, and one of the following recurrence relations:
\begin{equation}\label{1}
  \left\{\begin{aligned}& (k+2)(k+2\lambda+1)a_{k+2}=(k-n)(k+n+2\lambda+2\mu)a_{k}, &&{\rm if}\ n,k\ {\rm is\ even}, \\
& (k+1)(k+2\lambda+2)a_{k+2}=(k-n)(k+n+2\lambda+2\mu)a_{k}, &&{\rm if}\ n,k\ {\rm is\ odd}.
 \end{aligned}\right.
\end{equation}
It is well known that  $C_s^{(\mu,\lambda)},s=0,\dots,n$ form an
orthogonal basis for $\mathcal{P}_n$ with respect to the
generalized Gegenbauer weight
$w_{\lambda,\mu}(x)=|x|^{2\lambda}(1-x^2)^{\mu-\frac12}$ on $[-1,1]$.
Meanwhile, it has been proved in \cite{BeGa} that the generalized
Gegenbauer polynomials $C_n^{(\mu,\lambda)}$ is the unique polynomial
solution of differential-difference  equation system
\begin{equation}\label{2.4}
\Bigg\{\begin{aligned}&\  (1-x^2)\mathcal{D}_\lambda^2p-(2\mu+1)x\mathcal{D}_\lambda p+\lambda_n^2p=0,\\&\ p(-x)=(-1)^np(x),
 \end{aligned}
\end{equation}
with
\begin{equation}\label{2.5}
\lambda_n^2=\left\{n(n+2\lambda+2\mu)+2\lambda\mu[1-(-1)^n]\right\}
\end{equation}
in the meaning of a constant. In fact,  the symmetric property in \eqref{2.4} is not necessary, as stated in the following proposition.

\begin{prop}\label{prop1} Let $\lambda\ge0$. If the differential-difference  equation
\begin{equation}\label{2.4-3}
\mathcal{L}[\,p\,]:=(1-x^2)\mathcal{D}_\lambda^2p-(2\mu+1)x\mathcal{D}_\lambda
p+M^2p=0,
\end{equation}
has a nontrivial solution $p\in\mathcal{P}_n$ for some $M>0$, then
$p=cC_s^{(\mu,\lambda)}$ with $M=\lambda_s$ given by \eqref{2.5},
where $s=1,\dots,n$ and $c$ is an arbitrary nonzero constant.
\end{prop}
\begin{proof} For $1\le s\le n$, let $p(x)=\sum_{k=0}^sa_kx^k$ satisfy \eqref{2.4-3} with some $M>0$ and $a_s\neq 0$. Set $a_{s+1}=a_{s+2}=0$.
Then the straightforward calculation gives
$$\mathcal{D}_\lambda p=\sum_{k=1}^s\left\{k+\lambda\left[1-(-1)^k\right]\right\}a_kx^{k-1},$$
and
$$\mathcal{D}_\lambda^2p=\sum_{k=0}^s\left\{k+1+\lambda\left[1-(-1)^{k+1}\right]\right\}\left\{k+2+\lambda\left[1-(-1)^{k}\right]\right\}a_{k+2}x^{k},$$
which leads to
\begin{align*}
  \mathcal{L}[\,p\,]&=\sum_{k=0}^s\{\{k+1+\lambda[1-(-1)^{k+1}]\}\{k+2+\lambda[1-(-1)^{k}]\}a_{k+2}
\\&\quad+\{M^2-\{k+2\mu+\lambda[1-(-1)^{k-1}]\}\{k+\lambda[1-(-1)^k]\}\}a_k\}x^{k}.
\end{align*}
This yields that for all $k=0,\dots,s$,
$$\{k+1+\lambda[1-(-1)^{k+1}]\}\{k+2+\lambda[1-(-1)^{k}]\}a_{k+2}
=\{\{k+2\mu+\lambda[1-(-1)^{k-1}]\}\{k+\lambda[1-(-1)^k]\}-M^2\}a_k.$$
Since $a_s\neq0$ and $a_{s+1}=a_{s+2}=0$, we get  $M=\lambda_s$,
and
$$\left\{\begin{aligned} & a_{s-1}=a_{s-3}=\dots=0,\\ &(k+2)(k+2\lambda+1)a_{k+2}=(k-s)(k+s+2\lambda+2\mu)a_{k}, &&{\rm if}\ s,k\ {\rm is\ even}, \\
& (k+1)(k+2\lambda+2)a_{k+2}=(k-s)(k+s+2\lambda+2\mu)a_{k}, &&{\rm if}\ s,k\ {\rm is\ odd}.
 \end{aligned}\right.
 $$which, by \eqref{1}, means  $p=cC_s^{(\mu,\lambda)}$.
This completes the proof.
\end{proof}

We also remark that for $n$ being even, the generalized Gegenbauer polynomials $C_n^{(\mu,\lambda)}$ is the unique polynomial solution of differential equation
$$
(1-x^2)p''-(2\lambda+2\mu+1)xp'+\frac{2\lambda}xp'+n(n+2\lambda+2\mu)p=0
$$
in the meaning of a constant, which can be seen from the following proposition.
%the generalized Gegenbauer polynomials $C_n^{(\mu,\lambda)}$ is

%we also have the following result.
\begin{prop}\label{prop3-0} Let $\lambda>0$ and $\mu>-1/2$. If $n\ge2$ and the differential equation
\begin{equation}\label{2.4-0}
L[\,p\,]:=(1-x^2)p''-(2\lambda+2\mu+1)xp'+\frac{2\lambda}xp'+M^2p=0,
\end{equation}
has a nontrivial solution $p\in\mathcal{P}_n$ for some $M>0$, then
$p=cC_s^{(\mu,\lambda)}$ with $2\le s\le n$
 being even and $M^2=s(s+2\lambda+2\mu)$,  where
$c$ is an arbitrary nonzero constant.
\end{prop}
\begin{proof} For $2\le s\le n$, let $p(x)=\sum_{k=0}^sa_kx^k$ satisfy \eqref{2.4-0} with some $M>0$ and $a_{s}\neq0$. Set $a_{-1}=a_{s+1}=a_{s+2}=0$. Then
$$p'(x)=\sum_{k=1}^ska_kx^{k-1}\ \ {\rm and}\ \ p''(x)=\sum_{k=2}^sk(k-1)a_kx^{k-2},$$
which leads to
\begin{align*}
  (1-x^2)p''(x)&=\sum_{k=0}^{s-2}(k+2)(k+1)a_{k+2}x^{k}-\sum_{k=2}^sk(k-1)a_kx^{k}
\\&=\sum_{k=-1}^{s}\left[(k+2)(k+1)a_{k+2}-k(k-1)a_k\right]x^{k},
\end{align*}
and
\begin{align*}
  \frac{2\lambda}xp'(x)&=2\lambda\sum_{k=1}^ska_kx^{k-2}
  =2\lambda\sum_{k=-1}^{s}(k+2)a_{k+2}x^{k}.
\end{align*}Thus we obtain
$$L[\,p\,]=\sum_{k=-1}^{s}\Big\{(k+2)(k+2\lambda+1)a_{k+2}+\left[M^2-k(k+2\lambda+2\mu)\right]a_k\Big\}x^{k}.$$
which yields
$$(k+2)(k+2\lambda+1)a_{k+2}=\left[k(k+2\lambda+2\mu)-M^2\right]a_k,\ k=-1,0,1,\dots,s.$$
Since $a_s\neq0$ and $a_{s+2}=0$, we get $M^2=s(s+2\lambda+2\mu)$. It
follows from the equality $a_{-1}=0$  that $a_{1}=a_3=\dots=0$,
which yields that $s$ is even. We obtain
$$(k+2)(k+2\lambda+1)a_{k+2}=(k-n)(k+s+2\lambda+2\mu)a_k,\ k=0,2,4,\dots,s-2.
$$According to \eqref{1} we know  $p=cC_s^{(\mu,\lambda)}$.
This completes the proof.
\end{proof}

\subsection{Generalized Hermite polynomials}

\

For $ \kappa>-1$, the generalized Laguerre polynomials $L_n^\kappa,\, n\in\mathbb N$ are given by
$$L_n^\kappa(x):=\frac1{n!}x^{-\kappa}e^x\left(\frac{\rm d}{{\rm d}x}\right)^n(x^{n+\kappa}e^{-x}),$$
which is an orthogonal basis for $\mathcal{P}_n$  with respect to
the generalized Laguerre weight $x^\kappa e^{-x}$ on $\mathbb
R_+$.

Let $\lambda\ge0$.
The generalized Hermite polynomials $H_n^\lambda$ can be given by the generalized Laguerre polynomials. More precisely,
$$\Bigg\{\begin{aligned}&H_{2m}^\lambda(x)=(-1)^m2^{2m}m!\,L_m^{\lambda-1/2}(x^2),
\\&H_{2m+1}^\lambda(x)=(-1)^m2^{2m+1}m!\,xL_m^{\lambda+1/2}(x^2).\end{aligned}
$$
Moreover, it follows from \cite{BeGa} that the coefficients $a_{k},k=0,\dots,n$ of $H_n^\lambda$  satisfy $a_k=0$ if $n-k$ is odd, and one of the following recurrence relations:
\begin{equation}\label{2}
  \Bigg\{\begin{aligned}& (k+2)(k+2\lambda+1)a_{k+2}=2(k-n)a_{k}, &&{\rm if}\ n,k\ {\rm is\ even}, \\
& (k+1)(k+2\lambda+2)a_{k+2}=2(k-n)a_{k}, &&{\rm if}\ n,k\ {\rm is\ odd}.
 \end{aligned}
\end{equation}It is well known that $H_s^\lambda,s=0,\dots,n$ form an orthogonal basis for $\mathcal{P}_n$ with respect to
the generalized Hermite weight  $w_\lambda(x)=|x|^{2\lambda}e^{-x^2}$
on $\mathbb R$ . Meanwhile, it has been proved in \cite{BeGa} that
the generalized Hermite polynomial $H_n^\lambda$ is the unique polynomial
solution of the differential-difference  equation system
\begin{equation}\label{2-1}\Bigg\{\begin{aligned}
&\ \mathcal{D}_\lambda^2p-2 x\mathcal{D}_\lambda p+\lambda_n^2p=0,
\\&\  p(-x)=(-1)^np(x),
 \end{aligned}
 \end{equation}with
\begin{equation}\label{2.8}
\lambda_n^2=2\{n+\lambda[1-(-1)^n]\}
\end{equation}in the meaning of a constant. In fact, the symmetric property in \eqref{2-1} is not necessary, as stated in the following proposition.
\begin{prop}\label{prop2}Let $\lambda\ge0$. If the differential-difference  equation
\begin{equation}\label{2.4-2}
\mathcal{L}[\,p\,]:=\mathcal{D}_\lambda^2p-2 x\mathcal{D}_\lambda
p+M^2p=0
\end{equation}
has a nontrivial solution $p\in\mathcal{P}_n$ for some $M>0$, then
$p=cH_s^\lambda$ with $M=\lambda_s$ given by \eqref{2.8}, where
$s=1,\dots,n$ and $c$ is an arbitrary nonzero constant.
\end{prop}
\begin{proof} For $1\le s\le n$, let $p(x)=\sum_{k=0}^sa_kx^k$ satisfy \eqref{2.4-2} with some $M>0$ and $a_s\neq 0$.  Set $a_{s+1}=a_{s+2}=0$.
Then following the proof of Proposition \ref{prop1} we have
\begin{align*}
  \mathcal{L}[\,p\,]=&\sum_{k=0}^s\{\{k+1+\lambda[1-(-1)^{k+1}]\}\{k+2+\lambda[1-(-1)^{k}]\}a_{k+2}
+\{M^2-2\{k+\lambda[1-(-1)^k]\}\}a_k\}x^{k},
\end{align*}which yields that for all $k=0,\dots,s$,
$$\{k+1+\lambda[1-(-1)^{k+1}]\}\{k+2+\lambda[1-(-1)^{k}]\}a_{k+2}=\{2\{k+\lambda[1-(-1)^k]\}-M^2\}a_k.$$
Since $a_s\neq0$ and $a_{s+1}=a_{s+2}=0$, we get  $M=\lambda_s$,
and
$$\left\{\begin{aligned}&a_{s-1}=a_{s-3}=\dots=0,\\& (k+2)(k+2\lambda+1)a_{k+2}=2(k-s)a_{k},\ \ \  {\rm if}\ s,k\ {\rm is\ even}, \\
& (k+1)(k+2\lambda+2)a_{k+2}=2(k-s)a_{k},\ \ \ {\rm if}\ s,k\ {\rm is\ odd},
 \end{aligned}\right.
 $$which, by \eqref{2}, means  $p=cH_s^\lambda$.
This completes the proof.
\end{proof}

We also remark that for $n$ being even, the generalized Hermite polynomial $H_n^\lambda$ is the unique polynomial solution of differential equation
$$
p''-2xp'+\frac{2\lambda}xp'+2np=0
$$
in the meaning of a constant, which can be seen from the following proposition.

\begin{prop}\label{prop4-0} Let $\lambda>0$ and $\mu>-1/2$. If $n\ge2$ and the differential equation
\begin{equation}\label{2.4-1}
L[\,p\,]:=p''-2xp'+\frac{2\lambda}xp'+M^2p=0,
\end{equation}
has a nontrivial solution $p\in\mathcal{P}_n$ for some $M>0$, then
$p=cH_s^\lambda$ with $M^2=2s$ and $2\le s\le n$ being even, where
$c$ is an arbitrary nonzero constant.
\end{prop}
\begin{proof} For $2\le s\le n$, let $p(x)=\sum_{k=0}^sa_kx^k$ satisfy \eqref{2.4-1} with some $M>0$ and $a_s\neq 0$. Set $a_{-1}=a_{s+1}=a_{s+2}=0$.
Then following the proof of Proposition \ref{prop3-0} we have
$$L[\,p\,]=\sum_{k=-1}^{s}\Big\{(k+2)(k+2\lambda+1)a_{k+2}+(M^2-2k)a_k\Big\}x^{k},$$
which yields that for all $k=-1,0,1,\dots,s$,
$$(k+2)(k+2\lambda+1)a_{k+2}+(M^2-2k)a_k=0.$$Since $a_s\neq0$ and $a_{s+2}=0$, we get $M^2=2s$. It
follows from the equality $a_{-1}=0$  that $a_{1}=a_3=\dots=0$,
which yields that $s$ is even. We obtain  $$
 (k+2)(k+2\lambda+1)a_{k+2}=(2k-M^2)a_k,\ k=0,2,4,\dots,s-2.
$$According to \eqref{2} we know   $p=cH_s^{\lambda}$.
The proof is completed.
\end{proof}

\section{Proofs of Theorems \ref{thm2}-\ref{thm1}}

The proofs rely on the following duality relation between the extremal
problem
\begin{equation}\label{3.1}
  M_{n}\left(L_2(W_\lambda),\sqrt{A}\,\frac{\rm d}{{\rm d}x}\right)=\sup_{0\neq p\in\mathcal{P}_n}\frac{\|\sqrt{A}\,p'\|_{L_2(W_\lambda)}}{\|p\|_{L_2(W_\lambda)}},\ \lambda>0,
\end{equation} and the integral equation system
\begin{equation}\label{3.2} \int_{I}\left\{A(x)p''+C(x)p'+\frac{2\lambda}xp'+M^2p\right\}q\,W_\lambda=0,\ {\rm for\ all}\ q\in\mathcal P_n,
\end{equation} with some $M>0$, where $W_\lambda(x)$, $A(x)$, and $C(x)$ are given as in Table \ref{t1}.

\begin{lem}\label{lem1}
Let  $n\ge2$, $w\in CW_e:=\{(1-x^2)^{\mu-\frac12}, e^{-x^2}\}$,
$\lambda>0$, $\mu>-1/2$, and $W_\lambda(x):=|x|^{2\lambda}w(x)$.
Denote $M_{n}(W_\lambda)$ by the supremum of all those $M>0$ for
which the integral equation system \eqref{3.2} has a nontrivial
solution $p\in\mathcal{P}_n$. We have
$$M_{n}(W_\lambda)=M_{n}\left(L_2(W_\lambda),\sqrt{A}\,\frac{\rm d}{{\rm
d}x}\right),$$where $A(x)=1-x^2$ if $w(x)=(1-x^2)^{\mu-\frac12}$, and
$A(x)=1$ if $w(x)=e^{-x^2}$.
\end{lem}
\begin{proof} For abbreviation, we set $$M_n\equiv M_{n}\left(L_2(W_\lambda),
\sqrt{A}\,\frac{\rm d}{{\rm d}x}\right).$$
First we prove  $M_{n}\le M_n(W_{\lambda})$ in two cases.

{Case 1:} $W_\lambda=w_{\lambda,\mu}$. Let $p^*\in\mathcal{P}_n$
be an extremal polynomial of
$$M_{n}:=\sup_{0\neq p\in\mathcal{P}_n}\frac{\|\sqrt{1-x^2}\,p'\|_{L_2(w_{\lambda,\mu})}}{\|p\|_{L_2(w_{\lambda,\mu})}}.$$ Then the function
$$\varphi(t):=\frac{\|\sqrt{1-x^2}(p^*+tq)'\|_{L_2(w_{\lambda,\mu})}^2}{\|p^*+tq\|_{L_2(w_{\lambda,\mu})}^2},\ q\in\mathcal{P}_n,$$
has a maximum at $t=0$, which leads to $\varphi'(0)=0$. It follows
that
\begin{equation}\label{3.3}
\int_{-1}^1(1-x^2)p^*{'}q'w_{\lambda,\mu}=M_{n}^2\int_{-1}^1p^*q\,w_{\lambda,\mu},\ {\rm for\ all}\ q\in\mathcal{P}_n.
\end{equation} Meanwhile, integrating by parts shows that for all $p,q\in\mathcal{P}_n$,
\begin{equation}\label{3.4}
   \int_{-1}^1(1-x^2)p'q'\,w_{\lambda,\mu}=\int_{-1}^1\left[(2\lambda+2\mu+1)xp'-(1-x^2)p{''}(x)-\frac{2\lambda}xp'\right]q\,w_{\lambda,\mu},
\end{equation}
which, together with \eqref{3.3}, yields that \eqref{3.2} has a
nontrivial  solution $p^*\in \mathcal P_n$ with $M_n$ for the case
$W_\lambda=w_{\lambda,\mu}$. Then $M_{n}\le M_n(W_{\lambda})$.

{Case 2:} $W_\lambda=w_\lambda$. The proof is similar to the first
case. Let $p^*\in\mathcal{P}_n$ be an extremal polynomial of
$$M_{n}:=\sup_{0\neq p\in\mathcal{P}_n}\frac{\|p'\|_{L_2(w_{\lambda})}}{\|p\|_{L_2(w_{\lambda})}}.$$Then  the function
$$\varphi(t):=\frac{\|(p^*+tq)'\|_{L_2(w_{\lambda})}^2}{\|p^*+tq\|_{L_2(w_{\lambda})}^2},\ q\in\mathcal{P}_n,$$
has a maximum at $t=0$, which leads to $\varphi'(0)=0$. It follows
that
\begin{equation}\label{3.5}
\int_{\mathbb R}p^*{'}q'w_{\lambda}=M_{n}^2\int_{\mathbb
R}p^*q\,w_{\lambda},\ {\rm for\ all}\ q\in\mathcal{P}_n.
\end{equation}Meanwhile, integrating by parts leads that
\begin{equation}\label{3.6}
   \int_{\mathbb R}p'q'\,w_{\lambda}=\int_{\mathbb R}\left[2xp'-p''-\frac{2\lambda}xp'\right]q\,w_{\lambda},\ {\rm for\ all}\
   p,q\in\mathcal{P}_n,
\end{equation}which, together with \eqref{3.5},    yields that \eqref{3.2} has a
nontrivial  solution $p^*\in \mathcal P_n$ with $M_n$ for the case
$W_\lambda=w_{\lambda}$. Then $M_{n}\le M_n(W_{\lambda})$.

Now we turn to prove $M_{n}\ge M_n(W_{\lambda})$. Assume that
$p\in\mathcal{P}_n$ is a nontrivial solution of the integral
equation system \eqref{3.2} with some $M>0$. It follows from
\eqref{3.4} and \eqref{3.6} that for all $q\in\mathcal{P}_n$,
$$\int_Ip'q'\,AW_{\lambda}=\int_I\left\{-A(x)p''-C(x)p'-\frac{2\lambda}xp'\right\}q\,W_{\lambda}.$$
Applying the above equality with this $p$ and $q=p$, we obtain
\begin{align*}\int_{I}\left|\sqrt{A}\, p'\right|^2\,W_{\lambda}=\int_{I}\left\{-A(x)p{''}-C(x)p'-\frac{2\lambda}xp'\right\}p\,W_{\lambda}=M^2\int_{I}|p|^2\,W_{\lambda},
\end{align*}which gives $M_{n}\ge M_n(W_{\lambda})$.

Therefore, $M_{n}= M_n(W_{\lambda})$.
This completes the proof.
\end{proof}

Now we give the proofs of Theorems \ref{thm2} and \ref{thm1}.

\begin{proof}[Proof of Theorem \ref{thm2}]
\

  (i) For the case of $n=1$, we observe that the extremal polynomial must be attained by $p^*(x)=c(x+a_0)\ (c\neq0)$. Then the straightforward calculation  gives
$$\|p^*{'}\|_{L_2(w_{\lambda})}^2=c^2\int_{\mathbb R}w_{\lambda}=c^2\Gamma\left(\lambda+\frac12\right),$$
and
$$\|p^*\|_{L_2(w_{\lambda})}^2=c^2\int_{\mathbb R} w_{\lambda+1}+c^2a_0^2\int_{\mathbb R}w_{\lambda}=c^2\Gamma\left(\lambda+\frac12\right)\cdot\left({\lambda+\frac12}+a_0^2\right),$$
which leads that
\begin{align*}
M_{1}\left(L_2(w_{\lambda}),\frac{\rm d}{{\rm d}x}\right)
&=\sup_{a_0\in\mathbb
R}\frac{\|p^*{'}\|_{L_2(w_{\lambda})}}{\|p^*\|_{L_2(w_{\lambda})}}
=\sqrt{\frac{2}{2\lambda+1}},
\end{align*}with the extremal polynomial $p^*(x)=cx,\, c\neq0$.

  (ii)
For the case of $n\ge2$, let $p^*\in\mathcal{P}_n$ be an extremal polynomial of
$$
 M_{n}\left(L_2(w_{\lambda}),\frac{\rm d}{{\rm d}x}\right):=\sup_{0\neq p\in\mathcal{P}_n}\frac{\|p'\|_{L_2(w_{\lambda})}}{\|p\|_{L_2(w_{\lambda})}},
$$and $M_n(w_{\lambda})$ be the supremum of all those $M>0$ for
which the integral equation system
$$
\int_{\mathbb R}q\,L[\,p\,]\,w_{\lambda}=0,\  {\rm for\ all}\ q\in\mathcal{P}_n,
$$
 has a nontrivial
solution $p\in\mathcal{P}_n$, where $$L[\,p\,]:=p''-2xp'+\frac{2\lambda}xp'+M^2p.$$
By Lemma \ref{lem1}, we have to determine the value of $M_n(w_{\lambda})$.

   Consider first the case of even $n\ge2$. By Proposition \ref{prop4-0} we know
$L[\,H_n^\lambda\,]=0$ with $M^2=2n$, which implies
$$\int_{\mathbb R}q\,L[\,H_n^\lambda\,]\,w_{\lambda}=0,\  {\rm for\ all}\ q\in\mathcal{P}_n.$$
This deduces $M_n(w_{\lambda})\ge \sqrt{2n}$. Combining with \eqref{1.5-0}, we have
\begin{equation}\label{3.7}
M_{n}\left(L_2(w_{\lambda}),\frac{\rm d}{{\rm d}x}\right)\ge \sqrt{2n}>\sqrt{2n-2}>M_{n-1}\left(L_2(w_{\lambda}),\frac{\rm d}{{\rm d}x}\right).
\end{equation}
This  yields that  any extremal polynomial $p^*$ must be even. In
fact, assume that $p^*\in\mathcal P_n$ is an extremal polynomial.
We set
$$p^*=p_e^*+p_o^*,\ \ p^*_e(x)=\frac{p^*(x)+p^*(-x)}2,\
p^*_o(x)=\frac{p^*(x)-p^*(-x)}2,$$ where $p_e^*\in\mathcal{P}_n$
and $p_o^*\in\mathcal{P}_{n-1}$ are the even and odd parts of
$p^*$, respectively.  It follows that
\[\|p^*{'}\|_{L_2(w_{\lambda})}^2=\|p_e^*{'}\|_{L_2(w_{\lambda})}^2+\|p_o^*{'}\|_{L_2(w_{\lambda})}^2\]
and
\[\|p^*\|_{L_2(w_{\lambda})}=\|p_e^*\|_{L_2(w_{\lambda})}^2+\|p_o^*\|_{L_2(w_{\lambda})}^2.\]
If $p^*$ is not even, then $0\neq p^*_o\in \mathcal P_{n-1}$. By
\eqref{3.7} we have
$$M_{n}^2\left(L_2(w_{\lambda}),\frac{\rm d}{{\rm d}x}\right)=\frac{\|p^*{'}\|_{L_2(w_{\lambda})}^2}{\|p^*\|_{L_2(w_{\lambda})}^2}>
M^2_{n-1}\left(L_2(w_{\lambda}),\frac{\rm d}{{\rm d}x}\right)\ge
\frac{\|p_o^*{'}\|_{L_2(w_{\lambda})}^2}{\|p_o^*\|_{L_2(w_{\lambda})}^2},$$
which gives
$$\frac{\|p_e^*{'}\|_{L_2(w_{\lambda})}^2}{\|p_e^*\|_{L_2(w_{\lambda})}^2}>\frac{\|p^*{'}\|_{L_2(w_{\lambda})}^2}{\|p^*\|_{L_2(w_{\lambda})}^2}.$$
This contradicts the assumption that $p^*$ is   extremal. Thus
$p^*$ must be an even polynomial. Meanwhile, since $p^*$ is
extremal and even, we get $L[\,p^*\,]\in \mathcal P_n$, and
\begin{equation}\label{3.8}
\int_{\mathbb R}q\,L[\,p^*\,]\,w_{\lambda}=0,\  {\rm for\ all}\ q\in\mathcal{P}_n,
\end{equation}
with $M_n(w_{\lambda})$, which deduces that $L[\,p^*\,]=0$ by
setting $q=L[\,p^*\,]$ in \eqref{3.8}. By Proposition
\ref{prop4-0} we get $M_n(w_{\lambda})\le\sqrt{2n}$. Therefore,
$M_n(w_{\lambda})= \sqrt{2n}$ with the extremal polynomial
$H_n^\lambda$.

 Next consider the case of odd $n\ge2$.   By $\lambda>0$ and \eqref{1.5-0}  we have
$$M_{n}\left(L_2(w_{\lambda}),\frac{\rm d}{{\rm d}x}\right)>\sqrt{2n-\frac{4\lambda}{1+2\lambda}}>
\sqrt{2n-2}=M_{n-1}\left(L_2(w_{\lambda}),\frac{\rm d}{{\rm
d}x}\right).$$Using the same method as in the case of even $n$, we
show that the extremal polynomial $p^*$ must be odd. Hence, we can
write
 $p^*(x)=\sum_{j=0}^ma_{2j+1}x^{2j+1}$ with $m=(n-1)/2$ and $a_{n}\neq0$. We derive from \eqref{3.8} that
\begin{equation}\label{3.9}
\sum_{j=0}^ma_{2j+1}\langle x^{2i+1},L[\,x^{2j+1}\,]\rangle_{w_{\lambda}}=0,\ i=0,\dots,m.
\end{equation}
By the straightforward calculation, we have, for all $j=0,\dots,m$,
$$L[\,x^{2j+1}\,]=(2j+1)(2j+2\lambda)x^{2j-1}+(M^2-4j-2)x^{2j+1}.$$
This leads to
$$\langle x^{2i+1},L[\,x^{2j+1}\,]\rangle_{w_{\lambda}}=(2j+1)(2j+2\lambda)d_{2i+2j}+(M^2-4j-2)d_{2i+2j+2},$$
where  $d_{2s}=\int_{\mathbb
R}x^{2s}w_{\lambda}=\Gamma(s+\lambda+1/2)$ for all $s\in\mathbb
N_0$. Define
$$F_{m+1}(t):=\det\left(\left\{(2j+1)(2j+2\lambda)d_{2i+2j}+(t-4j-2)d_{2i+2j+2}\right\}_{i,j=0}^m\right).$$
Using the Crammer rule, we derive from $p^*\neq0$ that
$F_{m+1}(M^2)=0$. Therefore,
$M_n(w_{\lambda,\mu})=\sqrt{\nu_{m+1}}$ with the extremal
polynomials determined by \eqref{3.9} with $M^2=\nu_{m+1}$, where
$\nu_{m+1}$ is the large positive root of the polynomial
$F_{m+1}$. It also holds for $n=1$.

 The proof of  Theorem \ref{thm2} is completed.
\end{proof}

\vskip3mm
\begin{proof}[Proof of Theorem \ref{thm1}]
\

(i) For the case of $n=1$, we observe that the extremal polynomial
must be attained by $p^*(x)=c(x+a_0)\ (c\neq0)$. Then the
straightforward calculation  gives
$$\|\sqrt{1-x^2}\,p^*{'}\|_{L_2(w_{\lambda,\mu})}^2=c^2\int_{-1}^1w_{\lambda,\mu+1}=c^2B\left(\lambda+\frac12,\mu+\frac32\right),$$
and
\begin{align*}\|p^*\|_{L_2(w_{\lambda,\mu})}^2&=c^2\int_{-1}^1w_{\lambda+1,\mu}+c^2a_0^2\int_{-1}^1w_{\lambda,\mu}
=B\left(\lambda+\frac12,\mu+\frac32\right)\cdot\left(\frac{2\lambda+1}{2\mu+1}+c^2\right),
\end{align*}where
$B(\az,\beta)=\frac{\Gamma(\az)\Gamma(\beta)}{\Gamma(\az+\beta)}$
for $\az,\beta>0$. This leads that
\begin{align*}
M_{1}\left(L_2(w_{\lambda,\mu}),(1-x^2)^{\frac12}\,\frac{\rm
d}{{\rm d}x}\right) &=\sup_{a_0\in\mathbb
R}\frac{\|\sqrt{1-x^2}\,p^*{'}\|_{L_2(w_{\lambda,\mu})}}{\|p^*\|_{L_2(w_{\lambda,\mu})}}
=\sqrt{\frac{2\mu+1}{2\lambda+1}},
\end{align*}with the extremal polynomial $p^*(x)=cx,\ c\neq0$.

(ii) For the case of $n\ge2$, let $p^*\in\mathcal{P}_n$ be an
extremal polynomial of
$$
 M_n\equiv M_{n}\left(L_2(w_{\lambda,\mu}),(1-x^2)^{\frac12}\,\frac{\rm d}{{\rm d}x}\right)
 :=\sup_{0\neq
 p\in\mathcal{P}_n}\frac{\|\sqrt{1-x^2}\,p'\|_{L_2(w_{\lambda,\mu})}}{\|p\|_{L_2(w_{\lambda,\mu})}}.
$$By Lemma \ref{lem1} we have
\begin{equation}\label{3.10}
\int_{-1}^1q\,L_n[\,p^*\,]\,w_{\lambda,\mu}=0,\  {\rm for\ all}\
q\in\mathcal{P}_n,
\end{equation}
  where $$L_n[\,p\,]:=(1-x^2)p''-(2\lambda+2\mu+1)xp'+\frac{2\lambda}xp'+M_n^2p.$$
We write $p^*=p_e^*+p_o^*$, where $p_e^*$ and $p_o^*$ are the even
and odd parts of $p^*$, respectively. Then we have
 $$L_n[\,p^*\,]=L_n[\,p_e^*\,]+L_n[\,p_o^*\,],$$ where $L_n[\,p_e^*\,]\in \mathcal{P}_n$ is even and $L_n[\,p_o^*\,]$ is odd. Taking $q=L_n[\,p_e^*\,]$ leads to
$$\int_{-1}^1|L_n[\,p_e^*\,]|^2\,w_{\lambda,\mu}=0,$$which gives $L_n[\,p_e^*\,]=0$. Thus we have
\begin{equation}\label{3.11}
\int_{-1}^1q\,L_n[\,p_o^*\,]\,w_{\lambda,\mu}=0,\ {\rm for\ all}\
q\in\mathcal{P}_n.
\end{equation}
If $p_o^*=0$, then $p_e^*\neq 0$. Applying Proposition
\ref{prop3-0}, we deduce from $L_n[\,p_e^*\,]=0$ that
$p_e^*=cC_{2\lfloor n/2\rfloor}^{(\mu,\lambda)} \ (c\neq0)$ with
$$M_n^2=\max_{2\le {\rm even}\ s\le n}s(s+2\lambda+2\mu)=4\lfloor  n/2 \rfloor (\lfloor  n/2 \rfloor +\lz+\mu),$$ where $\lfloor x\rfloor$ denotes the largest
integer not exceeding $x$. If $p_o^*\neq 0$, then
$p_o^*(x)=\sum_{j=0}^ma_{2j+1}x^{2j+1}$ satisfy \eqref{3.11},
where $m=\lfloor (n-1)/2 \rfloor $. Then
$$ \sum_{j=0}^ma_{2j+1}\langle x^{2i+1},L_n[\,x^{2j+1}\,]\rangle_{w_{\lambda,\mu}}=0,\ i=0,\dots,m.$$
Meanwhile, by the straightforward calculation, we have, for all $j=0,\dots,m$,
$$L_n[\,x^{2j+1}\,]=(2j+1)(2j+2\lambda)x^{2j-1}+[M_n^2-(2j+1)(2j+2\lambda+2\mu+1)]x^{2j+1}.$$
This leads to
$$\langle x^{2i+1},L_n[\,x^{2j+1}\,]\rangle_{w_{\lambda,\mu}}=(2j+1)(2j+2\lambda)c_{2i+2j}+[M_n^2-(2j+1)(2j+2\lambda+2\mu+1)]c_{2i+2j+2},$$
where
$c_{2s}=\int_{-1}^1x^{2s}w_{\lambda,\mu}=B(s+\lambda+1/2,\mu+1/2)$
for all $s\in\mathbb N_0$. Define
$$G_{m+1}(t):=\det\left(\left\{(2j+1)(2j+2\lambda)c_{2i+2j}+[t-(2j+1)(2j+2\lambda+2\mu+1)]c_{2i+2j+2}\right\}_{i,j=0}^m\right).$$
Using the Crammer rule, we derive that $G_{m+1}(M^2)\neq0$ if and
only if $a_{2j+1}=0$ for all $j=0,\dots, m$. Since $p_o^*\neq 0$,
we get $G_{m+1}(M_n^2)=0$.   By Lemma \ref{lem1} we obtain that
$M_n=\sqrt{\nu_{m+1}}$, $\nu_{m+1}$ is the large positive root of
the polynomial $G_{m+1}$.  Therefore,
$$M_n(w_{\lambda,\mu})=\Bigg\{\begin{aligned}&\max\left\{\sqrt{\nu_{\frac{n}2}},\sqrt{n(n+2\lambda+2\mu)}\right\},\ &&{\rm if}\ n\ {\rm is\ even},\\ &\max\left\{\sqrt{\nu_{\frac{n+1}2}},\sqrt{(n-1)(n+2\lambda+2\mu-1)}\right\},\ &&{\rm if}\ n\ {\rm is\ odd},
\end{aligned}.$$ Clearly, it also holds for $n=1$.

  The proof of Theorem \ref{thm1} is completed.
\end{proof}

\vskip3mm In the final part of this section, let us calculate some
exact values  of the above two extremal problems in some special
cases.
\begin{exam} Consider the extremal problem $M_n\equiv M_{n}(L_2(w_\lambda),\frac{\rm d}{{\rm d}x})$, $\lambda>0$.

  (i) When $n=1$, it has been given in the proof of Theorem \ref{thm2} that
  $$M_{1}=\sqrt{\frac{2}{2\lambda+1}}<M_2=2,$$ with the extremal polynomials $p^*(x)=cx,
  \,c\neq0$.

  (ii) When $n=3$, we have
  \begin{align*}
    F_2(t)&=\begin{vmatrix}
2\lambda d_0+(t-2)d_2 & 6(1+\lambda)d_2+(t-6)d_4 \\
2\lambda d_2+(t-2)d_4 & 6(1+\lambda)d_4+(t-6)d_6
\end{vmatrix}
\\&= d_0d_2\left[(\lambda+\frac12)(\lambda+\frac32)t^2-2(2\lambda+3)(\lambda+1)t-3\lambda+9\right].
  \end{align*}Then the largest positive zero of $F_2(t)$ is
  $$\nu_2=\frac{8\lambda^2+20\lambda+12+2\sqrt{16\lambda^4+292\lambda^3+232\lambda^2+57\lambda+9}}{(2\lambda+1)(2\lambda+3)}.$$
   Therefore, we obtain
   $$M_{3}
=\sqrt{\frac{8\lambda^2+20\lambda+12+2\sqrt{16\lambda^4+292\lambda^3+232\lambda^2+57\lambda+9}}{(2\lambda+1)(2\lambda+3)}}.$$

\end{exam}

\vskip3mm

\begin{exam} Consider the extremal problem $M_n\equiv M_{n}(L_2(w_{\lambda,\mu}),(1-x^2)^{\frac12}\frac{\rm d}{{\rm d}x})$, $\lambda>0$, $\mu>-1/2$.

(i) When $n=1$, it has been given in the proof of Theorem
\ref{thm1} that
$$M_{1}=\sqrt{\frac{2\mu+1}{2\lambda+1}},
$$with the extremal polynomial $p^*(x)=cx,\,c\neq 0$.

(ii) When $n=2$, we have
  $$G_1(t)=(t-2\lambda-2\mu-1)c_2+2\lambda c_0$$has the largest positive zero
  $$\nu_1=\frac{2\mu+1}{2\lambda+1}<2(2\lambda+2\mu+2),$$ which
  gives $$M_{2}=\sqrt{2(2\lambda+2\mu+2)}.$$

(iii) When $n=3$ or $n=4$, we have
  \begin{align*}
    G_2(t)&=\begin{vmatrix}
2\lambda c_0+[t-(2\lambda+2\mu+1)]c_2 & 6(1+\lambda)c_2+[t-3(2\lambda+2\mu+1)]c_4 \\
2\lambda c_2+[t-(2\lambda+2\mu+1)]c_4 & 6(1+\lambda)c_4+[t-3(2\lambda+2\mu+1)]c_6
\end{vmatrix}
\\&= c_0c_2\begin{vmatrix}
2\lambda +[t-(2\lambda+2\mu+1)]\frac{\lambda+1/2}{\lambda+\mu+1} & 6(1+\lambda)+[t-3(2\lambda+2\mu+1)]\frac{\lambda+3/2}{\lambda+\mu+2} \\
-\frac{4\lambda (2\mu+1)}{(2\lambda+2\mu+2)(2\lambda+2\mu+4)}& [t-3(2\lambda+2\mu+1)]\frac{\lambda+3/2}{\lambda+\mu+2}\frac{2(2\mu+1)}{(2\lambda+2\mu+4)(2\lambda+2\mu+6)}
\end{vmatrix}
  \end{align*}Then we  obtain the following results as stated in Table \ref{t2}. (Here, in Table \ref{t2}, the symbol ``$\times$'' means that $G_2(t)$ has no positive roots.)
\begin{table}[ht]
\centering\label{t2}
\caption{The largest positive zero of $G_2(t)$.}
\begin{tabular}{|c| c| c|c||c|c |c|c|}
  \hline
  $\lambda(\lambda=-\mu)$&$\nu_2$& $M_{3}$&$M_4$ & $\lambda(\lambda=\mu)$& $\nu_2$&$M_{3}$&$M_4$\\
  \hline
  $ 0.4$ &$7.7460$ &$2.7832$&$4$&$4$& $28.1733$&$ 6$&$4\sqrt{5}$\\
  \hline
  $ 0.3$ &$ 7.1730$&$2.6782$&$4$ &$3$& $19.7266$ & $2\sqrt{7}$&$8$ \\
  \hline
  $ 0.2$ &$6.4061$&$2.5310$&$4$ &$2$ & $9.0000$&$2\sqrt{5}$&$4\sqrt{3}$\\
  \hline
  $0.1$ &$ 5.2820$&$2.2983$&$4$&$1$ &$\times$&$2\sqrt{3}$&$4\sqrt{2}$\\
  \hline\hline
  $\lambda(\lambda=\mu+1)$&$\nu_2$& $M_{3}$&$M_4$ & $\lambda(\lambda=\mu+10)$& $\nu_2$&$M_{3}$&$M_4$\\
  \hline
  $ 100$ &$800.9852$ &$28.3017$ &$2\sqrt{402}$&$40$& $ 484.5768$&$22.0131$&$24$\\
  \hline
  $ 50$ &$  400.9707$&$ 20.0243$&$12\sqrt{7}$ &$30$&$438.3382$&$20.9365$ &$20.9365$   \\
  \hline
  $ 10$ &$80.8660$&$8.9926$ &$2\sqrt{42}$&$20$ &$403.0921$&$20.0772$& $20.0772$\\
  \hline
  $1$ &$ 8.0494$&$ 2.8371$ &$2\sqrt{6}$ &$10$ &$387.1007$&$19.6749$&$19.6749$\\
  \hline
\end{tabular}
\end{table}
The above results implies that  the extremal problem
$M_{n}(L_2(w_{\lambda,\mu}),(1-x^2)^{\frac12}\tfrac{\rm d}{{\rm
d}x})$ with $\lambda>0$ is very complicated.
\end{exam}

\section{Proofs of Theorems \ref{thm4}-\ref{thm3}}

The proofs are based on the duality relation between the
extremal problem
\[
  M_{n}\left(L_2(W_\lambda),\sqrt{A}\,\mathcal{D}_\lambda\right)=\sup_{0\neq p\in\mathcal{P}_n}\frac{\|\sqrt{A}\,\mathcal{D}_\lambda p\|_{L_2(W_\lambda)}}{\|p\|_{L_2(W_\lambda)}}
\] and the differential equation
\begin{equation}\label{3.12}
A(x)\mathcal{D}_\lambda^2p+B(x)\mathcal{D}_\lambda p+M^2p=0,\ p\in\mathcal{P}_n
\end{equation}with some $M>0$, where $W_\lambda(x)$, $A(x)$, and $B(x)$ are given as in Table \ref{t1}.
\begin{lem}\label{lem2}
Let $\lambda\ge0$, $\mu>-1/2$, $w\in CW_e:=\{(1-x^2)^{\mu-\frac12},
e^{-x^2}\}$, and $W_\lambda(x):=|x|^{2\lambda}w(x)$. Denote by
$\mathcal{M}_{n}(W_\lambda)$ the supremum of all those $M>0$ for
which the differential equation \eqref{3.12} has a nontrivial
solution $p\in\mathcal{P}_n$. Then for all $n\in\mathbb N$,
$$\mathcal{M}_{n}(W_\lambda)=M_{n}\left(L_2(W_\lambda),\sqrt{A}\,\mathcal{D}_\lambda\right),$$
where $A(x)=1-x^2$ if $w(x)=(1-x^2)^{\mu-\frac12}$, and
$A(x)=1$ if $w(x)=e^{-x^2}$.
\end{lem}
\begin{proof}
First we prove that $$\mathcal{M}_{n}(W_\lambda)\ge M_{n}\equiv
M_{n}\left(L_2(W_\lambda),\sqrt{A}\,\mathcal{D}_\lambda\right)$$
in two cases.

{ Case 1:} $W_\lambda=w_{\lambda,\mu}$. In this case,
it reduces to prove that
$$\mathcal{M}_{n}(w_{\lambda,\mu})\ge M_n\equiv M_{n}\left(L_2(w_{\lambda,\mu}),(1-x^2)^{\frac12}\,\mathcal{D}_\lambda\right),$$
where $\mathcal{M}_{n}(w_{\lambda,\mu})$ denotes by the supremum
of all those $M>0$ for which the differential equation
\[
(1-x^2)\mathcal{D}_\lambda^2p-(2\mu+1) x\mathcal{D}_\lambda p+M^2p=0
\]has a nontrivial solution $p\in\mathcal{P}_n$, and
\begin{equation}\label{3.13}
M_{n}=\sup_{0\neq
p\in\mathcal{P}_n}\frac{\|\sqrt{1-x^2}\,\mathcal{D}_\lambda
p\|_{L_2(w_{\lambda,\mu})}}{\|p\|_{L_2(w_{\lambda,\mu})}}.
\end{equation}
%Now we prove $\mathcal{M}_n(w_{\lambda,\mu}) \ge M_{n}$.
For our purpose, let $p^*\in\mathcal{P}_n$ be an  extremal
polynomial of \eqref{3.13}. It follows that the function
$$\varphi(t):=\frac{\|\sqrt{1-x^2}\,\mathcal{D}_\lambda(p^*+tq)\|_{L_2(w_{\lambda,\mu})}^2}{\|p^*+tq\|_{L_2(w_{\lambda,\mu})}^2},\ q\in\mathcal{P}_n,$$
has a maximum at $t=0$. This leads to $\varphi'(0)=0$, which follows
\begin{equation}\label{3.14}
\int_{-1}^1(1-x^2)\mathcal{D}_\lambda p{^*}\mathcal{D}_\lambda
q\,w_{\lambda,\mu}=M_n^2\int_{-1}^1p^*q\,w_{\lambda,\mu},\ {\rm for\ all}\
q\in\mathcal{P}_n.
\end{equation}
We shall calculate $\int_{-1}^1(1-x^2)\mathcal{D}_\lambda p\,\mathcal{D}_\lambda q\,w_{\lambda,\mu}$ for any $p,q\in\mathcal{P}_n$.
By \eqref{1.8} we have
\begin{align}\label{3.15}
 &\int_{-1}^1(1-x^2)\mathcal{D}_\lambda p\,\mathcal{D}_\lambda q\,w_{\lambda,\mu}\notag
 \\=&\int_{-1}^1(1-x^2)q'\mathcal{D}_\lambda p\,w_{\lambda,\mu}
+\lambda\int_{-1}^1(1-x^2)\sigma(q)\,\mathcal{D}_\lambda p\,w_{\lambda,\mu}\notag
\\=&\int_{-1}^1(1-x^2)q'\mathcal{D}_\lambda p\,w_{\lambda,\mu}
+\lambda\int_{-1}^1(1-x^2)p'\sigma(q)\,w_{\lambda,\mu}\notag
+\lambda^2\int_{-1}^1(1-x^2)\sigma(p)\,\sigma(q)\,w_{\lambda,\mu}\notag
\\=&:I_1+\lambda I_2+\lambda^2 I_3.
\end{align}
For the integral $I_1$, by integrating by part, $(1-x^2)w_{\lambda,\mu}|_{-1}^1=0$, \eqref{1.8}, and $\sigma^2=0$, we have
\begin{align*}
 I_1&=(1-x^2)q\mathcal{D}_\lambda p\,w_{\lambda,\mu}\Big|_{-1}^1
 -\int_{-1}^1q\left[(1-x^2)\mathcal{D}_\lambda p\,w_{\lambda,\mu}\right]'
 \\&=-\int_{-1}^1q\left[(1-x^2)(\mathcal{D}_\lambda p)'-(2\mu+1)x\mathcal{D}_\lambda p+\frac{2\lambda}x(1-x^2)\mathcal{D}_\lambda p\right]w_{\lambda,\mu}
  \\&=-\int_{-1}^1q\left[(1-x^2)\mathcal{D}_\lambda^2p-\lambda(1-x^2)\sigma(\mathcal{D}_\lambda p)-(2\mu+1)x\mathcal{D}_\lambda p+\frac{2\lambda}x(1-x^2)\mathcal{D}_\lambda p\right]w_{\lambda,\mu}
  \\&=\int_{-1}^1q\left[(2\mu+1)x\mathcal{D}_\lambda p-(1-x^2)\mathcal{D}_\lambda^2p+\lambda(1-x^2)\sigma(\mathcal{D}_\lambda p)-\frac{2\lambda}x(1-x^2)\mathcal{D}_\lambda p\right]w_{\lambda,\mu}
  \\&=\int_{-1}^1q\left[(2\mu+1)x\mathcal{D}_\lambda p-(1-x^2)\mathcal{D}_\lambda^2p\right]\,w_{\lambda,\mu}+\lambda\int_{-1}^1q\left[\sigma(p')-2\frac{p'}x \right](1-x^2)\,w_{\lambda,\mu}
\\&\quad-2\lambda^2\int_{-1}^1q\frac{\sigma(p)}x(1-x^2)\,w_{\lambda,\mu}
  \\&=\int_{-1}^1q\left[(2\mu+1)x\mathcal{D}_\lambda p-(1-x^2)\mathcal{D}_\lambda^2p\right]\,w_{\lambda,\mu}-\lambda\int_{-1}^1q\frac{p'(x)+p'(-x)}x(1-x^2)\,w_{\lambda,\mu}
\\&\quad-2\lambda^2\int_{-1}^1q\frac{\sigma(p)}x(1-x^2)\,w_{\lambda,\mu}.
\end{align*}It follows from \eqref{2.2} and \eqref{2.3} that
\begin{align*}
 \int_{-1}^1q\frac{p'(x)+p'(-x)}x(1-x^2)\,w_{\lambda,\mu}=\int_{-1}^1p'\,\sigma(q)(1-x^2)\,w_{\lambda,\mu}=I_2,
\end{align*}
and
\begin{align*}
2\int_{-1}^1q\frac{\sigma(p)}x(1-x^2)\,w_{\lambda,\mu}=\int_{-1}^1\sigma(p)\sigma(q)(1-x^2)\,w_{\lambda,\mu}=I_3,
\end{align*}which leads to
\begin{align}\label{3.16}
   I_1&=\int_{-1}^1q\left[(2\mu+1)x\mathcal{D}_\lambda p-(1-x^2)\mathcal{D}_\lambda^2p\right]\,w_{\lambda,\mu}-\lambda I_2- \lambda^2I_3.
\end{align}
From \eqref{3.15} and \eqref{3.16} we have
\begin{align}\label{3.17}
  \int_{-1}^1(1-x^2)\mathcal{D}_\lambda p\,\mathcal{D}_\lambda q\,w_{\lambda,\mu}=\int_{-1}^1q\left[(2\mu+1)x\mathcal{D}_\lambda p-(1-x^2)\mathcal{D}_\lambda^2p\right]w_{\lambda,\mu},\ {\rm for\ all}\ p,q\in\mathcal{P}_n.
\end{align}
Then combining \eqref{3.14} with \eqref{3.17} for $p=p^*$, we obtain
$$\int_{-1}^1q\left[(1-x^2)\mathcal{D}_\lambda^2 p^*-(2\mu+1)x\mathcal{D}_\lambda p^*+M_n^2p^*\right]w_{\lambda,\mu}=0,\ {\rm for\ all}\ q\in\mathcal{P}_n.$$
Since $$(1-x^2)\mathcal{D}_\lambda^2
p^*-(2\mu+1)x\mathcal{D}_\lambda p^*+M_n^2p^*\in\mathcal{P}_n,$$
by the orthogonality of polynomials it immediately implies
\begin{equation*}
(1-x^2)\mathcal{D}_\lambda^2 p^*-(2\mu+1)x\mathcal{D}_\lambda
p^*+M_n^2p^*=0,
\end{equation*}which gives $ \mathcal{M}_{n}(w_{\lambda,\mu})\ge M_n$.

%Now we turn to verify that $M_n\ge\mathcal{M}_{n}(w_{\lambda,\mu})$ holds. Let $p\in\mathcal{P}_n$ be a nontrivial solution of \eqref{3.16} with some $M>0$. Applying \eqref{3.21} with this $p$ and $q=p$, we obtain
%\begin{align*}\int_{-1}^1\left|\sqrt{1-x^2}\,\mathcal{D}_\lambda p\right|^2\,w_{\lambda,\mu}=\int_{-1}^1p\left[(2\mu+1)x\mathcal{D}_\lambda p-(1-x^2)\mathcal{D}_\lambda^2 p\right]w_{\lambda,\mu}=M^2\int_{-1}^1|p|^2\,w_{\lambda,\mu},\end{align*}which implies  $\mathcal{M}_{n}(w_{\lambda,\mu})\le M_{n}$.
%Hence $\mathcal{M}_{n}(w_{\lambda,\mu})=M_n$.
{Case 2:} $W_\lambda=w_\lambda$. The proof is similar to the first
case.  In this case, it reduces to prove
$$\mathcal{M}_{n}(w_{\lambda})\ge M_n \equiv  M_{n}\left(L_2(w_{\lambda}),\mathcal{D}_\lambda\right),$$
where $\mathcal{M}_{n}(w_{\lambda})$ denotes by the supremum of
all those $M>0$ for which the differential equation
\begin{equation*}
\mathcal{D}_\lambda^2p-2x\mathcal{D}_\lambda p+M^2p=0\end{equation*}has a nontrivial solution $p\in\mathcal{P}_n$, and
\begin{equation}\label{3.18}
M_{n}=\sup_{0\neq
p\in\mathcal{P}_n}\frac{\|\mathcal{D}_\lambda
p\|_{L_2(w_{\lambda})}}{\|p\|_{L_2(w_{\lambda})}}.
\end{equation} Let $p^*\in\mathcal{P}_n$ be an extremal polynomial of \eqref{3.18}.
It follows that  the function
$$\varphi(t):=\frac{\|\mathcal{D}_\lambda(p^*+tq)\|^2}{\|(p^*+tq)\|^2},\ q\in\mathcal{P}_n,$$
has a maximum at $t=0$. This leads to $\varphi'(0)=0$, which
follows that
\begin{equation}\label{3.19}
\int_{\Bbb R}\mathcal{D}_\lambda p{^*}\mathcal{D}_\lambda
q\,w_{\lambda}=M_{n}^2\int_{\Bbb R}p^*q\,w_{\lambda}.
\end{equation}

We shall calculate $\int_{\Bbb R}\mathcal{D}_\lambda
p\,\mathcal{D}_\lambda q\,w_{\lambda}$ for any
$p,q\in\mathcal{P}_n$. By \eqref{1.8} we have
\begin{align}\label{4.8}
 \int_{\Bbb R}\mathcal{D}_\lambda p\,\mathcal{D}_\lambda q\,w_{\lambda}
 &=\int_{\Bbb R}q'\mathcal{D}_\lambda p\,w_{\lambda}
+\lambda\int_{\Bbb R}\sigma(q)\,\mathcal{D}_\lambda
p\,w_{\lambda}\notag
\\&=\int_{\Bbb R}q'\mathcal{D}_\lambda p\,w_{\lambda}
+\lambda\int_{\Bbb R}p'\sigma(q)\,w_{\lambda} +\lambda^2\int_{\Bbb
R}\sigma(p)\,\sigma(q)\,w_{\lambda}\notag
\\&=:I_1+\lambda I_2+\lambda^2 I_3.
\end{align}
For the integral $I_1$, by integrating by part,
$w_{\lambda}|_{-\infty}^\infty=0$, \eqref{1.8}, and $\sigma^2=0$,
we have
\begin{align*}
 I_1&=q\mathcal{D}_\lambda p\,w_{\lambda}\Big|_{-\infty}^\infty
 -\int_{\Bbb R}q\left(\mathcal{D}_\lambda p\,w_{\lambda}\right)'
 \\&=-\int_{\Bbb R}q\left[(\mathcal{D}_\lambda p)'-2x\mathcal{D}_\lambda p+\frac{2\lambda}x\mathcal{D}_\lambda p\right]w_{\lambda}
  \\&=-\int_{\Bbb R}q\left[\mathcal{D}_\lambda^2p-\lambda\sigma(\mathcal{D}_\lambda p)-2x\mathcal{D}_\lambda p+\frac{2\lambda}x\mathcal{D}_\lambda p\right]w_{\lambda}
  \\&=\int_{\Bbb R}q\left[2x\mathcal{D}_\lambda p-\mathcal{D}_\lambda^2p\right]\,w_{\lambda}-\int_{\Bbb R}q\left[-\lambda\sigma(\mathcal{D}_\lambda p)+\frac{2\lambda}x\mathcal{D}_\lambda p\right]w_{\lambda}
  \\&=\int_{\Bbb R}q\left[2x\mathcal{D}_\lambda p-\mathcal{D}_\lambda^2p\right]\,w_{\lambda}+\lambda\int_{\Bbb R}q\left[\sigma(p')-2\frac{p'}x \right]w_{\lambda}-2\lambda^2\int_{\Bbb R}q\frac{\sigma(p)}xw_{\lambda}
  \\&=\int_{\Bbb R}q\left[2x\mathcal{D}_\lambda p-\mathcal{D}_\lambda^2p\right]\,w_{\lambda}-\lambda\int_{\Bbb R}q\frac{p'(x)+p'(-x)}x\,w_{\lambda}-2\lambda^2\int_{\Bbb R}q\frac{\sigma(p)}xw_{\lambda}.
\end{align*}It follows from \eqref{2.2} and \eqref{2.3} that
\begin{align*}
 \int_{\Bbb R}q\frac{p'(x)+p'(-x)}x\,w_{\lambda}=\int_{\Bbb R}p'\,\sigma(q)\,w_{\lambda}=I_2,
\end{align*}
and
\begin{align*}
2\int_{\Bbb R}q\frac{\sigma(p)}x\,w_{\lambda}=\int_{\Bbb
R}\sigma(p)\sigma(q)\,w_{\lambda}=I_3,
\end{align*}which leads to
\begin{align}\label{4.9}
   I_1&=\int_{\Bbb R}q\left[2x\mathcal{D}_\lambda p-\mathcal{D}_\lambda^2p\right]w_{\lambda}-\lambda I_2- \lambda^2I_3.
\end{align}
From \eqref{4.8} and \eqref{4.9} we have
\begin{align}\label{3.20}
  \int_{\Bbb R}\mathcal{D}_\lambda p\,\mathcal{D}_\lambda q\,w_{\lambda}=\int_{\Bbb R}q\left[2x\mathcal{D}_\lambda p-\mathcal{D}_\lambda^2p\right]w_{\lambda},\ \forall\, p,q\in\mathcal{P}_n.
\end{align}

Combining with \eqref{3.19} and \eqref{3.20} for $p=p^*$, we obtain
$$\int_{\mathbb R}\left[\mathcal{D}_\lambda p^*-2x\mathcal{D}_\lambda p^*+M_n^2p^*\right]q\,w_{\lambda}=0,\  {\rm for\ all}\ q\in\mathcal{P}_n.$$By the orthogonality of polynomials it immediately leads to
\begin{equation*}
\mathcal{D}_\lambda^2 p^*-2x\mathcal{D}_\lambda p^*+M_n^2p^*=0,
\end{equation*}which gives $\mathcal{M}_{n}(w_{\lambda})\ge M_n$.

Next we  prove  $\mathcal{M}_{n}(W_{\lambda})\le M_n$. Assume that
$p\in\mathcal{P}_n$ is a nontrivial solution of the differential
equation \eqref{3.12} with some $M>0$. By \eqref{3.17} and
\eqref{3.20} we have, for all $q\in\mathcal{P}_n$,
$$\int_{I}\mathcal{D}_\lambda p\,\mathcal{D}_\lambda q\,AW_{\lambda}=
\int_{I}\left[-B(x)\mathcal{D}_\lambda
p-A(x)\mathcal{D}_\lambda^2p\right]q\,W_{\lambda}.$$Applying the
above equality with this $p$ and $q=p$, we obtain
\begin{align*}
  \quad\int_{I}|\mathcal{D}_\lambda p|^2\,AW_{\lambda}
  &=\int_I\left[-B(x)\mathcal{D}_\lambda p-A(x)\mathcal{D}_\lambda^2 p\right]p\,W_{\lambda}
  =M^2\int_I|p|^2\,W_{\lambda},
\end{align*}which implies that $\mathcal{M}_{n}(W_{\lambda})\le M_n$.
The proof of Lemma \ref{lem2} is completed.
\end{proof}

\vskip3mm
Now we turn to prove Theorems \ref{thm4} and \ref{thm3}.

\begin{proof}[Proof of Theorem \ref{thm4}]

\

By  Lemma \ref{lem2},  Proposition \ref{prop2}, and \eqref{2.8}
 we obtain
\begin{align*}M_{n}^2(L_2(w_\lambda),\mathcal{D}_\lambda)&=\left\{\begin{aligned}&\max\{2n,2(n+2\lambda-1)\},&&{\rm if}\ n\ {\rm is\ even},
  \\&\max\{2(n+2\lambda),2(n-1)\},&&{\rm if}\ n\ {\rm is\ odd},\end{aligned}\right.
 \\&=\left\{\begin{aligned}&2n,&&{\rm if}\ n\ {\rm is\ even\ and}\ 0\le\lambda\le1/2,
 \\&2(n+2\lambda-1),&&{\rm if}\ n\ {\rm is\ even\ and}\ \lambda>1/2,
  \\&2(n+2\lambda),&&{\rm if}\ n\ {\rm is\ odd}.\end{aligned}\right.
\end{align*}Then  applying Proposition \ref{prop2}, we have the following results stated in Table \ref{t3} with $c\neq0$.
\begin{table}[!htbp]
\centering \caption{Extremal polynomials with
$\lambda\ge0$.}\label{t3}
\begin{tabular}{|c| c| c| c|}
  \hline
  $n$& $\lambda$ &$M_{n}^2(L_2(w_\lambda),\mathcal{D}_\lambda)$& ${\rm extremal\ polynomials}$\\
  \hline
  ${\rm even}$& $[0,1/2]$ & $2n$ & $cH_n^\lambda$  \\
  $ $& $(1/2,+\infty)$ & $2(n+2\lambda-1)$ & $cH_{n-1}^\lambda$\\
  \hline
   ${\rm odd}$& $[0,+\infty)$ &$2(n+2\lambda)$& $cH_n^\lambda$\\
  \hline
\end{tabular}
\end{table}

This finishes the proof of Theorem \ref{thm4}.
\end{proof}
\vskip3mm
\begin{proof}[Proof of Theorem \ref{thm3}]

\

For $n\ge3$ being odd, it follows from \eqref{2.5} that
\begin{align*}
  \lambda_n^2-\lambda_{n-1}^2
  &=n(n+2\lambda+2\mu)+4\lambda\mu -(n-1)(n-1+2\lambda+2\mu)
  \\&=2n+2\lambda+2\mu+4\lambda\mu-1
  \\&=2n+(2\lambda+1)(2\mu+1)-2\ge4,
\end{align*}
which, by Lemma \ref{lem2} and Proposition \ref{prop1}, leads to
$$M_{n}^2\left(L_2(w_{\lambda,\mu}),(1-x^2)^{\frac12}\mathcal{D}_\lambda\right)=\sqrt{n(n+2\lambda+2\mu)+4\lambda\mu }.$$
It also holds for $n=1$.

 For $n\ge 2$ being even, it follows from
\eqref{2.5} that
\begin{align*}
  \lambda_n^2-\lambda_{n-1}^2
  &=n(n+2\lambda+2\mu) -(n-1)(n-1+2\lambda+2\mu)-4\lambda\mu
  \\&=2n+2\lambda+2\mu-4\lambda\mu-1=2n-(2\lambda-1)(2\mu-1).
\end{align*}
Then by Lemma \ref{lem2} and Proposition \ref{prop1} , if
$(2\lambda-1)(2\mu-1)\le4$, then $\lambda_n^2-\lambda_{n-1}^2\ge
0$, and
$$M_{n}^2\left(L_2(w_{\lambda,\mu}),(1-x^2)^{\frac12}\mathcal{D}_\lambda\right)=n(n+2\lambda+2\mu);$$
if $(2\lambda-1)(2\mu-1)>4$ then
\begin{align*}M_{n}^2\left(L_2(w_{\lambda,\mu}),(1-x^2)^{\frac12}\mathcal{D}_\lambda\right)&=\left\{\begin{aligned}&n(n+2\lambda+2\mu),&&{\rm if}\ n\ge n_0,
  \\&(n-1)(n-1+2\lambda+2\mu)+4\lambda\mu,&&{\rm if}\ n< n_0,\end{aligned}\right.
\\&=\left\{\begin{aligned}&n(n+2\lambda+2\mu),&&{\rm if}\ n\ge n_0,
  \\&n(n+2\lambda+2\mu)+2(n_0-n),&&{\rm if}\ n< n_0,\end{aligned}\right.
\end{align*}where $n_0:=(\lambda-1/2)(2\mu-1)$.
Combining with Proposition \ref{prop1} we have the following
results stated in Table \ref{t4} with $c\neq0$.

\begin{table}[!h]
\centering \caption{Extremal polynomials with $\lambda\ge0$ and
$\mu>-1/2$.}\label{t4}
\begin{tabular}{|c| c| c |c|}
  \hline
  $n$&$\lambda,\mu$ & $M_{n}^2\left(L_2(w_{\lambda,\mu}),(1-x^2)^{\frac12}\mathcal{D}_\lambda\right)$& ${\rm extremal\ polynomials}$\\
  \hline
  $ {\rm even}$ &$(2\lambda-1)(2\mu-1)>4$ &$n(n+2\lambda+2\mu)\ (n\ge n_0)$ & $cC_n^{(\mu,\lambda)}$  \\

  $ $ &$ $ &$n(n+2\lambda+2\mu)+2(n_0-n)\ (n<n_0)$ & $cC_{n-1}^{(\mu,\lambda)} $  \\

  $ $ &$(2\lambda-1)(2\mu-1)\le4$ &$n(n+2\lambda+2\mu)$ & $cC_n^{(\mu,\lambda)}$  \\
  \hline
  ${\rm odd}$ &$ $&$n(n+2\lambda+2\mu)+4\lambda\mu$& $cC_n^{(\mu,\lambda)}$\\
  \hline
\end{tabular}
\end{table}

This finishes the proof of of Theorem \ref{thm3}.
\end{proof}

\section{Proof of Theorem \ref{thm5}}

\begin{proof}[Proof of Theorem \ref{thm5}]

\

 For our purpose, it suffices to show that the inequality
\begin{align}\label{3.21}\notag
  &\quad(2\lambda_n^2+B'(0))\|\sqrt{A}\,\mathcal{D}_\lambda
  p\|_{L_2(W_\lambda)}^2+ 2\lambda B'(0)\langle Ap',p'(-\cdot)\rangle_{W_\lambda}
  \\&\le 2\lambda^3B'(0)\|\sqrt{A}\,\sigma(p)\|_{L_2(W_\lambda)}^2
  + \lambda_n^4\|p\|_{L_2(W_\lambda)}^2+\|A\,\mathcal{D}_\lambda^2p\|_{L_2(W_\lambda)}^2,
\end{align}
holds for all $p\in\mathcal{P}_n$, where $W_\lambda(x)$, $A(x)$, and $B(x)$ are given in the Table \ref{t1}.

Now we prove \eqref{3.21}. For simplicity, we denote $\|\cdot\|\equiv \|\cdot\|_{L_2(W_\lambda)}$ and $\langle\cdot,\cdot\rangle\equiv\langle\cdot,\cdot\rangle_{W_\lambda}$.
For any $p\in\mathcal{P}_n$, let
$$\mathcal{L}_n[\,p\,]:=A(x)\mathcal{D}_\lambda^2p+B(x)\mathcal{D}_\lambda p+\lambda_n^2p.$$
Then we obtain
\begin{align}\label{3.22}
  \|\mathcal{L}_n[\,p\,]\|^2
  =&\|B\,\mathcal{D}_\lambda p\|^2+2\langle A\,\mathcal{D}_\lambda^2 p,B\,\mathcal{D}_\lambda p\rangle
  \\&+2\lambda_n^2[\langle A\,\mathcal{D}_\lambda^2 p,p\rangle+\langle B\,\mathcal{D}_\lambda p,p\rangle]+\lambda_n^4\|p\|^2+\|A\,\mathcal{D}_\lambda^2p\|^2
  .\notag
\end{align}
On the one hand, by \eqref{2.1}, \eqref{1.8}, $B(x)=B'(0)x$,
$AW_{\lambda}|_{\partial I}=0$, and $\sigma^2=0$, we have
\begin{align*}
  &\|B\mathcal{D}_\lambda p\|^2
  =\int_IB(\mathcal{D}_\lambda p)^2|x|^{2\lambda}(Aw)'
  \\=&\,B(\mathcal{D}_\lambda p)^2AW_\lambda\Big|_{\partial I}-B'(0)\int_I(\mathcal{D}_\lambda p)^2AW_\lambda
  -2\int_I(\mathcal{D}_\lambda p)'(\mathcal{D}_\lambda p)ABW_\lambda-2\lambda\int_I\frac{(\mathcal{D}_\lambda p)^2}xABW_\lambda
  \\=&-B'(0)\|\sqrt{A}\,\mathcal{D}_\lambda p\|^2
  -2\int_I(\mathcal{D}_\lambda^2p)(\mathcal{D}_\lambda p)ABW_\lambda
  +2\lambda\int_I\sigma(\mathcal{D}_\lambda p)(\mathcal{D}_\lambda p)ABW_\lambda-2\lambda B'(0)\int_I(\mathcal{D}_\lambda p)^2AW_\lambda
  \\=&-B'(0)\|\sqrt{A}\,\mathcal{D}_\lambda p\|^2
  -2\langle A\,\mathcal{D}_\lambda^2 p,B\,\mathcal{D}_\lambda p\rangle
  +2\lambda\int_I\sigma(p')(\mathcal{D}_\lambda p)ABW_\lambda-2\lambda B'(0)\|\sqrt{A}\mathcal{D}_\lambda p\|^2.
\end{align*}
By \eqref{2.2}  we get
\begin{align*}
  &\int_I\sigma(p')(\mathcal{D}_\lambda p)ABW_\lambda
  \\=&\int_Ip'\sigma(p')ABW_\lambda+\lambda\int_I\sigma(p')\sigma( p)ABW_\lambda
   \\=&\int_I\frac{(p')^2}xABW_\lambda-\int_I p'\frac{p'(-\cdot)}xABW_\lambda
   +2\lambda\int_Ip'\frac{\sigma(p)}xABW_\lambda
   \\=&B'(0)\int_I(\mathcal D_\lz p-\lz \sigma(p))^2AW_\lambda-B'(0)\int_Ip'p'(-\cdot)AW_\lambda+2\lambda B'(0)\int_Ip'\sigma( p)AW_\lambda
   \\=&B'(0)\|\sqrt{A}\mathcal{D}_\lambda p\|^2+\lambda^2B'(0)\|\sqrt{A}\,\sigma(p)\|^2-B'(0)\langle Ap',p'(-\cdot)\rangle.
\end{align*}
Then we derive
\begin{align}\label{3.23}
  2\langle A\,\mathcal{D}_\lambda^2 p,B\,\mathcal{D}_\lambda p\rangle+\|B\mathcal{D}_\lambda p\|^2
  =&-B'(0)\|\sqrt{A}\,\mathcal{D}_\lambda p\|^2+2\lambda^3B'(0)\|\sqrt{A}\,\sigma(p)\|^2
\\&-2\lambda B'(0)\langle Ap',p'(-\cdot)\rangle.\notag
\end{align}
On the other hand, by \eqref{1.8} and $\sigma^2=0$ we obtain
\begin{align*}
  \langle A\,\mathcal{D}_\lambda^2 p, p\rangle
  &=\int_Ip(\mathcal{D}_\lambda p)'\,A W_\lambda
  +\lambda\int_Ip\sigma(\mathcal{D}_\lambda p)AW_\lambda
  \\&=\int_Ip(\mathcal{D}_\lambda p)'AW_\lambda
  +\lambda\int_Ip\sigma(p')\,AW_\lambda.
\end{align*}
Integrating by part, $AW_{\lambda}|_{\partial I}=0$, \eqref{1.8},
and \eqref{2.1} imply
\begin{align*}
  &\int_Ip\,(\mathcal{D}_\lambda p)' AW_\lambda
  \\=&\ p\mathcal{D}_\lambda p\, W_\lambda\Big|_{\partial I}
  -\int_I\mathcal{D}_\lambda p\left(pAW_\lambda\right)'
 \\=&-\int_Ip'\mathcal{D}_\lambda p\,AW_\lambda
 -\int_Ip\,\mathcal{D}_\lambda p(Aw)'|x|^{2\lambda}
 -2\lambda\int_Ip\frac{\mathcal{D}_\lambda p}xAW_\lambda
  \\=&-\int_I(\mathcal{D}_\lambda p)^2\,AW_\lambda+\lambda\int_I\sigma(p)\,\mathcal{D}_\lambda p\,AW_\lambda
  -\int_Ip\,\mathcal{D}_\lambda p\,BW_\lambda-2\lambda\int_Ip\frac{\mathcal{D}_\lambda p}xAW_\lambda
  \\=&-\|\sqrt{A}\,\mathcal{D}_\lambda p\|^2+\lambda\int_Ip'\sigma(p)\,AW_\lambda+\lambda^2\int_I\sigma(p)^2\,AW_\lambda-\langle B\mathcal{D}_\lambda p,p\rangle
  \\&-2\lambda\int_I p'\frac{p}xAW_\lambda
-2\lambda^2\int_I p\frac{\sigma(p)}xAW_\lambda
  \\=&-\|\sqrt{A}\,\mathcal{D}_\lambda p\|^2+\lambda\int_Ip'\frac{p(x)-p(-x)}x\,AW_\lambda
  -\langle B\mathcal{D}_\lambda p,p\rangle-2\lambda\int_Ip'\frac{p}xAW_\lambda
  \\=&-\|\sqrt{A}\,\mathcal{D}_\lambda p\|^2-\lambda\int_Ip'\frac{p(x)+p(-x)}x\,AW_\lambda
  -\langle B\mathcal{D}_\lambda p,p\rangle
  \\=&-\|\sqrt{A}\,\mathcal{D}_\lambda p\|^2-\lambda\int_Ip\,\sigma(p')\,AW_\lambda
  -\langle B\mathcal{D}_\lambda p,p\rangle,
\end{align*}where in the fifth
and the last equalities we used \eqref{2.2} and \eqref{2.3},
respectively. Then
\begin{align}\label{3.24}
  \langle A\,\mathcal{D}_\lambda^2 p, p\rangle+\langle B\mathcal{D}_\lambda p,p\rangle
  &=-\|\sqrt{A}\,\mathcal{D}_\lambda p\|^2.
\end{align}

Therefore, it follows from \eqref{3.22}, \eqref{3.23}, and \eqref{3.24} that
\begin{align*}
  \|\mathcal{L}_n[\,p\,]\|^2
  =&\|A\,\mathcal{D}^2_\lambda p\|^2-B'(0)\|\sqrt{A}\,\mathcal{D}_\lambda p\|^2
  -2\lambda_n^2\|\sqrt{A}\,\mathcal{D}_\lambda p\|^2
  \\&+\lambda_n^4\|p\|^2
  +2\lambda^3B'(0)\|\sqrt{A}\,\sigma(p)\|^2-2\lambda B'(0)\langle Ap',p'(-\cdot)\rangle\ge0.
\end{align*}Hence \eqref{3.21} holds, with the equality  if and only if $\mathcal{L}_n[\,p\,]=0$, which, by Propositions \ref{prop1} and \ref{prop2}, completes the proof of Theorem \ref{thm5}.
\end{proof}

\noindent\textbf{Acknowledgments}
  The authors  were
supported by the National Natural Science Foundation of China
(Project no. 12371098).

\end{document}